\documentclass[11pt]{amsart}

\usepackage[utf8]{inputenc} 			
\usepackage[T1]{fontenc} 			
\usepackage[francais, english]{babel} 			
\usepackage{layout}					
\usepackage{setspace}				
\usepackage{graphics}				
\usepackage{cite}

\usepackage[hyperindex,breaklinks]{hyperref}

\usepackage{amsthm}
\usepackage{amsmath}
\usepackage{amssymb}
\usepackage{mathrsfs}
\usepackage{dsfont}			
\usepackage{enumerate}

\usepackage{wasysym}

\usepackage{mathtools,amscd}            
\usepackage{tikz-cd}            

\usepackage{hyperref}		
\hypersetup{
    colorlinks,
    citecolor=blue,
    filecolor=blue,
    linkcolor=blue,
    urlcolor=blue
}

\usepackage{geometry}				

\theoremstyle{definition}
\newtheorem{df}{Definition}[section]

\newtheorem{ex}[df]{Example}
\theoremstyle{theorem}
\newtheorem{prop}[df]{Proposition}		
\newtheorem{thm}[df]{Theorem}
\newtheorem{lm}[df]{Lemma} 
\newtheorem{cor}[df]{Corollary} 
\newtheorem{fact}[df]{Fact} 
\newtheorem{question}{Question}

\theoremstyle{remark}
\newtheorem*{claim}{Claim}
\usepackage{etoolbox}
\AtEndEnvironment{proof}{\setcounter{claimn}{0}}
\newtheorem{rk}[df]{Remark}

\title{Forking, Imaginaries and other features of $\ACFG$}

\author{Christian d'Elb\'ee}\thanks{Intitut Camille Jordan, Université Lyon 1.}

\date{\small\today}

\begin{document}

\newcommand\R{\mathbb{R}}		
\newcommand\C{\mathbb{C}}
\newcommand\Q{\mathbb{Q}}
\newcommand\SSS{\mathbb{S}^3}
\newcommand\N{\mathbb{N}}
\newcommand\U{\mathcal{U}}
\newcommand\F{\mathbb{F}}
\newcommand\B{\mathbb{B}}
\newcommand\T{\mathbb{T}}
\newcommand\K{\mathbb{K}}
\newcommand\A{\mathbb{A}}
\newcommand\M{\mathbb{M}}

\newcommand\xx{\bar x}
\newcommand\yy{\bar y}
\newcommand\zz{\bar z}

\newcommand\LL{\mathscr{L}}
\newcommand\UU{\mathscr{U}}
\newcommand\PP{\mathscr{P}}
\newcommand\ZZ{\mathscr{Z}}
\newcommand\MM{\mathscr{M}}
\newcommand\NN{\mathscr{N}}
\newcommand\KK{\mathscr{K}}
\newcommand\CC{\mathfrak{C}}
\newcommand\CCC{\mathcal{C}}
\newcommand\Aa{\mathfrak{A}}
\newcommand\GG{G_{\bar a}}
\newcommand\GGG{G_{\bar a '}}

\newcommand\Z{\mathbb{Z}}
\newcommand\Ha{\mathbb{H}}
\newcommand\MH{\mathcal{M}_\mathbb{H}}
\newcommand\Haz{\mathbb{H}_{\Z}}
\newcommand\HH{\widehat{\mathbb{H}_{\Z}}}
\newcommand\UN{\mathds{1}}

\newcommand\LR[1]{\mathbb{L}^{#1}(\mathbb{R})}		

\newcommand\Pcal{\mathcal{P}}
\newcommand\Cq{\mathcal{C}_{\Q}}
\newcommand\Rea{\mathscr{R}}

\newcommand{\cRM}[1]{\MakeUppercase{\romannumeral #1}}
\newcommand{\cRm}[1]{\textsc{\romannumeral #1}}	
\newcommand{\crm}[1]{\romannumeral #1}

\newcommand\LLr{\mathscr{L}_{\mathrm{ring}}}
\newcommand\LLf{\mathscr{L}_{\mathrm{field}}}
\newcommand\CCCC{\mathscr{C}}
\newcommand\oFp{\overline{\F_p}}
\newcommand\OO{\mathscr{O}}

\newcommand{\ol}[1]{\overline{{#1}}}
\newcommand{\set}[1]{\left\{ {#1} \right\}}
\newcommand{\vect}[1]{\langle {#1} \rangle}
\newcommand{\abs}[1]{\lvert {#1} \rvert}
\newcommand{\intinf}[1]{\lfloor {#1} \rfloor}
\newcommand{\intsup}[1]{\lceil {#1} \rceil}

\newcommand{\scal}[2]{\left\langle {#1} \middle| {#2} \right\rangle}	

\newcommand{\factor}[2]{{\left.\raisebox{.2em}{$#1$}\middle/\raisebox{-.2em}{$#2$}\right.}}

\def\Ind#1#2{#1\setbox0=\hbox{$#1x$}\kern\wd0\hbox to 0ex{\hss$#1\mid$\hss}
\lower.9\ht0\hbox to 0ex{\hss$#1\smile$\hss}\kern\wd0}
\def\Notind#1#2{#1\setbox0=\hbox{$#1x$}\kern\wd0\hbox to 0ex{\mathchardef\nn="0236\hss$#1\nn$\kern1.4\wd0\hss}\hbox to 0ex{\hss$#1\mid$\hss}\lower.9\ht0
\hbox to 0ex{\hss$#1\smile$\hss}\kern\wd0}

\def\ind{\mathop{\mathpalette\Ind{}}}

\def\nind{\mathop{\mathpalette\Notind{}}}

\def\indi#1{\mathop{\ \ \hbox to 0ex{\hss$\vert^{\hbox to 0ex{$\scriptstyle#1$\hss}}$\hss}
\lower1ex\hbox to 0ex{\hss$\smile$\hss}\ \ }}

\def\nindi#1{\mathop{\ \ \hbox to 0ex{\hss$\!\not{\vert}^{\hbox to 0ex{$\scriptstyle\,#1$\hss}}$\hss}
\lower1ex\hbox to 0ex{\hss$\smile$\hss}\ \ }}

\newcommand{\findep}[1][]{%
  \mathrel{
    \mathop{
      \vcenter{
        \hbox{\oalign{\noalign{\kern-.3ex}\hfil$\vert$\hfil\cr
              \noalign{\kern-.7ex}
              $\smile$\cr\noalign{\kern-.3ex}}}
      }
    }\displaylimits_{#1}
  }
}

\newcommand{\nfindep}[1][]{%
  \mathrel{
    \mathop{
      \vcenter{
	\hbox{\oalign{\noalign{\kern-.3ex}\hfil$\!\not{\vert}$\hfil\cr
              \noalign{\kern-.7ex}
              $\smile$\cr\noalign{\kern-.3ex}}}
      }
    }\displaylimits_{#1}
  }
}

\makeatletter
\newcommand{\setword}[2]{%
  \phantomsection
  #1\def\@currentlabel{\unexpanded{#1}}\label{#2}%
}
\makeatother

\newcommand{\NSOP}[1]{\mathrm{NSOP}_{#1}}  
\newcommand{\ACFG}{\mathrm{ACFG}}
\newcommand{\NFCP}{\mathrm{NFCP}}
\newcommand{\ACF}{\mathrm{ACF}}
\newcommand{\RCF}{\mathrm{RCF}}
\newcommand{\PAC}{\mathrm{PAC}}
\newcommand{\PACG}{\mathrm{PACG}}
\newcommand{\Psf}{\mathrm{Psf}}
\newcommand{\DCF}{\mathrm{DCF}}
\newcommand{\SCF}{\mathrm{SCF}}
\newcommand{\ACVF}{\mathrm{ACVF}}
\newcommand{\ACFA}{\mathrm{ACFA}}
\newcommand{\acl}{\mathrm{acl}}
\newcommand{\dcl}{\mathrm{dcl}}
\newcommand{\cl}{\mathrm{cl}}
\newcommand{\Sg}{\mathrm{Sg}}
\newcommand{\di}{\mathrm{dim}}
\newcommand{\expo}{\mathrm{exp}}
\newcommand{\dg}{\mathrm{deg}}

\newsavebox{\auteurbm}
\newenvironment{cit}[1]{\small\slshape
  \savebox{\auteurbm}{\upshape\sffamily#1}
\begin{flushright}}{\\[4pt]\usebox{\auteurbm}
\end{flushright}\normalsize\upshape}

\let\oldabstract\abstract
\let\oldendabstract\endabstract
\makeatletter
\renewenvironment{abstract}
{\renewenvironment{quotation}%
               {\list{}{\addtolength{\leftmargin}{8em} 
                        \listparindent 1.5em%
                        \itemindent    \listparindent%
                        \rightmargin   \leftmargin%
                        \parsep        \z@ \@plus\p@}%
                \item\relax}%
               {\endlist}%
\oldabstract}
{\oldendabstract}
\makeatother

\newcommand{\red}[1]{\textcolor{red}{#1}}

\newcommand{\blue}[1]{\textcolor{blue}{#1}}

\maketitle				

\begin{abstract}
We study the generic theory of algebraically closed fields of fixed positive characteristic with a predicate for an additive subgroup, called $\ACFG$. This theory was introduced in~\cite{dE18A} as a new example of $\NSOP{1}$ non simple theory. In this paper we describe more features of $\ACFG$, such as imaginaries. We also study various independence relations  in $\ACFG$, such as Kim-independence or forking independence, and describe interactions between them.
\end{abstract}

\hrulefill
\tableofcontents	
\hrulefill

\clearpage
\section*{Introduction}

The theory of algebraically closed fields of fixed positive characteristic with a predicate for an additive subgroup admits a model companion, $\ACFG$~\cite{dE18A}. Unlike other generic expansions of $\ACF_p$, such as $\ACFA$ or the expansion by a generic predicate~\cite{CP98}, $\ACFG$ is $\NSOP 1$ and not simple. The study of $\NSOP 1$ theories has been rekindled due to the recent success in developing a Kim-Pillay style characterization (Chernikov and Ramsey~\cite{CR16}) and a geometric theory based on the notions of Kim-forking and Kim-independence (Kaplan and Ramsey~\cite{KR17}). Various examples of strictly $\NSOP 1$ theories appear since then. Among them are
\textit{
\begin{enumerate}
\item Generic $\LL$-structure $T^{\emptyset}_\LL$~\cite{KrR18};
\item Generic $K_{n,m}$-free bipartite graphs~\cite{CKr17};
\item omega-free PAC fields~\cite{C02}.
\end{enumerate}}
$\ACFG$ shares many features with those three archetypical examples. Our example appears to be slightly more complicated than \textit{(1)} and \textit{(2)}, due to the lack of weak elimination of imaginaries and its more algebraic aspect, which makes it closer to \textit{(3)}. Throughout this paper, we will point out both the similarities and the differences between those four examples, in order to emphasize what might be typical of $\NSOP 1$ theories.

We intend to give a description of $\ACFG$ based on the study of various independence relations in models of $\ACFG$. In Section~\ref{sec_gen}, we give basic properties of $\ACFG$. A notion of \emph{weak independence} (following the denomination of~\cite{C02}) was already described in~\cite{dE18A}, and shown to coincide with Kim-independence over models. We prove here that it satisfies all the properties of the Kim-Pillay characterization of simple theories~\cite{KP97} except one: \emph{base monotonicity}. This phenomenon, not predicted by~\cite{KR17} is similar to the case of \textit{(2)}. We define \emph{strong independence} --a similar notion appears in \textit{(1)}, \textit{(2)} and \textit{(3)}-- and we show that it lacks only one property of the Kim-Pillay characterization of simple theories: \emph{local character}. We give some structural properties of models of $\ACFG$, and prove that there exists generic subgroups of $\overline{\F_p}$, even more, almost all (in the sense of Baire) subgroups of $\oFp$ are generic, see Subsection~\ref{sub_baire}. 

Section~\ref{sec_im} and Section~\ref{sec_fork} can be read independently.

Section~\ref{sec_im} is dedicated to the description of imaginaries in a given model $(K,G)$ of $\ACFG$. The weak independence has a "dual" definition in the expanded structure $(K,G,K/G)$ which turns out to be easier to grasp than its original definition in $(K,G)$ (see Subsection~\ref{subsec_dual}). We extend the weak independence in $(K,G,K/G)$ and this allows us to mimic the classical argument that appears for instance in~\cite{CP98},~\cite{BMP17}, and~\cite{KrR18} to prove that $(K,G,K/G)$ has weak elimination of imaginaries.

In Section~\ref{sec_fork}, we describe forking in $\ACFG$. The strong independence plays a key role to show that forking equals dividing (for types). We also advertise some nice phenomena that appear when one forces the \emph{base monotonicity} property on a given independence relation. It seems that a general method for proving that dividing equals forking for types is arising from different examples, see Subsection~\ref{sec_forkdiscuss} 

The diagram in Figure~\ref{fig_diag1} represents the interactions between the independence relations that appears in models of $\ACFG$ and links them with usual independences. All arrows are strict, from that point of view, $\ACFG$ differs from \textit{(1)}, \textit{(2)} and \textit{(3)}, see Subsection~\ref{sec_forkdiscuss}.\\

~
\vfill
~
\begin{figure}
  \begin{center}
    \includegraphics{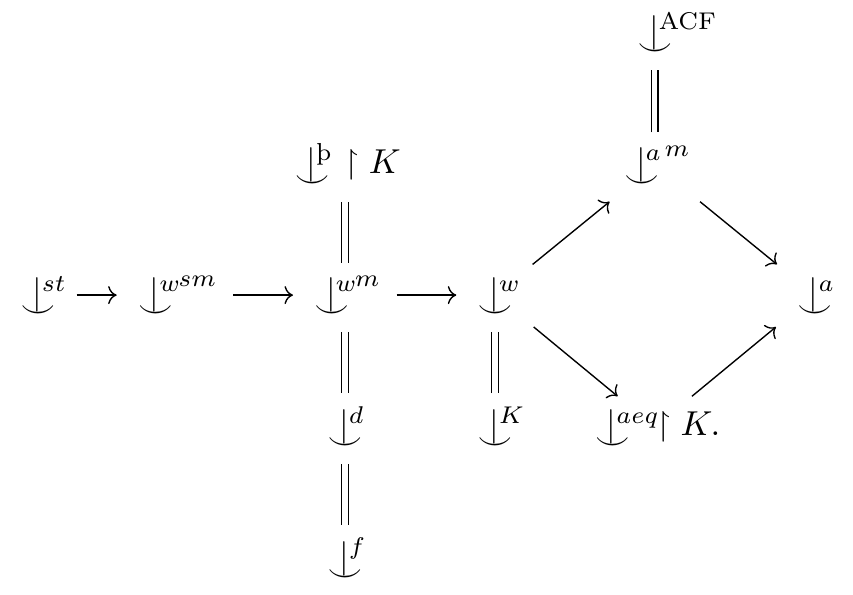}
  \end{center}
\caption{Interactions of independence relations in $\ACFG$.}\label{fig_diag1}
\end{figure}
~
\vfill
~

\clearpage

\noindent\textbf{Conventions and notations.} Capital letters $A,B,C$ stands for sets whereas small latin letters $a,b,c$ designate either singletons, finite or infinite tuples. 

We often identify tuples and sets when dealing with independent relations, for some tuple $a=a_1,\dots$ then $c\ind_C a$ has the same meaning as $c\ind_C \set{a_1,\dots}$.  Let $\ind, \indi 0$ be two ternary relations. We say that $\ind$ \emph{is stronger than} $\indi 0$ (or $\indi 0$ \emph{is weaker than} $\ind$) if for all $a,b,C$ we have $a\ind_C b \implies a \indi 0_C b$. We denote it by $\ind\rightarrow \indi 0$. For a prime $q$, $\F_{q^n}$ is the field with $q^n$ elements.

Throughout this paper numerous notions of independence will appear.
The following independence relations are defined in every theory.
 \begin{enumerate}
 \item $a\indi a _C b$ if and only if $acl(Ca)\cap acl(Cb) = acl(C)$
 \item $a\indi{aeq} _C b$ if and only if $acl^{eq}(Ca)\cap acl^{eq}(Cb) = acl^{eq}(C)$
  \item $a\indi{\text{\thorn}}_C b$ if and only if $tp(a/Cb)$ does not thorn-fork over $C$
  \item $a\indi{K}_C b$ if and only if $tp(a/Cb)$ does not Kim-fork over $C$
 \item $a\indi{d}_C b$ if and only if $tp(a/Cb)$ does not divide over $C$ 
 \item $a\indi{f}_C b$ if and only if $tp(a/Cb)$ does not fork over $C$
 \item $a\indi u _C b$ if and only if $tp(a/Cb)$ is finitely satisfiable in $C$
 \end{enumerate}
  Here is a list of properties for a ternary relation $\ind$ defined over small subsets of $\M$ a big model of some countable theory $T$ (in the last case the property is defined with respect to another ternary relation $\ind'$, also defined over small subsets of $\M$).
\begin{itemize}
  \item \setword{\bsc{Invariance}}{INV}. If $ABC\equiv A'B'C'$ then $A\ind_C B$ if and only if $A'\ind_{C'} B'$.
  \item \setword{\bsc{Finite Character}}{FIN}. If $a\ind_C B$ for all finite $a\subseteq A$, then $A\ind_C B$.
  \item \setword{\bsc{Symmetry}}{SYM}. If $A\ind_C B$ then $B\ind_C A$.
  \item \setword{\bsc{Closure}}{CLO} $A\ind_C B$ if and only if $A \ind_{\acl(C)} \acl(BC)$.
  \item \setword{\bsc{Monotonicity}}{MON}. If $A\ind_C BD$ then $A\ind_C B$.
  \item \setword{\bsc{Base Monotonicity}}{BMON}. If $A\ind_C BD$ then $A\ind_{CD} B$.
  \item \setword{\bsc{Transitivity}}{TRA}. If $A \ind_{CB} D$ and $B\ind_C D$ then $AB\ind_C D$.
  \item \setword{\bsc{Existence}}{EX}. For any $C$ and $A$ we have $A\ind_{C} C$.
  \item \setword{\bsc{Full Existence}}{EXT}. For all $A,B$ and $C$ there exists $A'\equiv_C A$ such that $A'\ind_C B$.
  \item \setword{\bsc{Extension}}{EXT2}. If $A\ind_C B$, then for all $D$ there exists $A'\equiv_{CB}A$ and $A'\ind_C BD$.
  \item \setword{\bsc{Local Character}}{LOC}. For all finite tuple $a$ and infinite $B$ there exists $B_0\subset B$ with $\abs{B_0}\leq \aleph_0$ and $a\ind_{B_0} B$.
  \item \setword{\bsc{Strong Finite Character}}{STRFINC} over $E$. If $a\nind_E b$, then there is a formula $\Lambda(x,b,e)\in tp(a/Eb)$ such that for all $a'$, if $a'\models \Lambda(x,b,e)$ then $a'\nind_E b$.
\item \setword{\bsc{Stationnarity}}{STAT} over $E$. If $c_1 \equiv_E c_2$ and $c_1\ind_E A$, $c_2\ind_E A$ then $c_1\equiv_{EA} c_2$.
\item \setword{\bsc{Witnessing}}{WIT}. Let $a,b$ be tuples, $\MM$ a model and assume that $a\nind_{\MM} b$. Then there exists a formula $\Lambda(x,b)\in tp(a/\MM b)$ such that for any global extension $q(x)$ of $tp(b/\MM)$ finitely satisfiable in $\MM$ and for any $(b_i)_{i<\omega}$ such that for all $i<\omega$ we have  $b_i\models q\upharpoonright \MM b_{<i}$, the set
 $\set{ \Lambda(x, b_i) \mid i<\omega}$ is inconsistent.
 \item $\ind '$-\setword{\bsc{amalgamation}}{AM} over $E$. If there exists tuples $c_1,c_2$ and sets $A,B$ such that
\begin{itemize}
\item $c_1\equiv_E c_2$
\item $A\ind'_E B$
\item $c_1\ind_E A$ and $c_2\ind_C B$
\end{itemize}
then there exists $c\ind_E A,B$ such that $c\equiv_A c_1$, $c \equiv_B c_2$, $A\indi a _{Ec} B$, $c\indi a _{EA} B$ and $c\indi a _{EB} A$.
\end{itemize}
If $A\ind _C B$, the set $C$ is called the \emph{base set}.
\clearpage

\section{Generalities on $\ACFG$}\label{sec_gen}

Let $p>0$ be a fixed prime number. Unless stated otherwise, every field we consider has characteristic $p$. Let $\LLr$ be the language of rings and $\LL_G = \LLr\cup\set{G}$ for $G$ a unary predicate. Let $\ACF_G$ be the $\LL_G$-theory whose models are algebraically closed fields of characteristic $p$ in which $G$ is a predicate for an additive subgroup. Let $\ACFG$ be the model companion of $\ACF_G$, see~\cite[Example 5.11]{dE18A}.

\subsection{Preliminaries, axioms and types}\label{sub_tp} 

The following is~\cite[Proposition 5.4]{dE18A}.
\begin{fact}[Axiomatisation of $\ACFG$]\label{prop_axACFG}
  The theory $\ACFG$ is axiomatised by adding to $\ACF_{G}$ the following $\LL_G$-sentences, for all tuples of variables $x'\subset x$, $y' \subset y$ and $\LLr$-formula $\phi(x,y)$
$$\forall y (\vect{y'}\cap G = \set{0} \wedge \theta_{\phi}(y)) \rightarrow (\exists x \phi(x,y)\wedge \vect{xy'}\cap G = \vect{x'}),$$
where $\theta_\phi(y)$ such that $K\models \theta_\phi(b)$ if and only if in an elementary extension of $K$, there exists a tuple of realisations of $\phi(x,b)$ which is $\F_p$-linearly independent over $K$ (see~\cite[Theorem 5.2]{dE18A}).
\end{fact}

By \cite[Proposition 1.15]{dE18A} we have the following, for $(K,G)\models\ACFG$ sufficiently saturated, and $a,b, C$ in $K$
\begin{enumerate}
    \item $\acl_{\ACFG}(C) = \acl_{\ACF}(C) =: \overline C$;
    \item $a\equiv_C b$ if and only if there exists an $\LL_G$-isomorphism $\sigma : \overline{Ca}\rightarrow \overline{Cb}$ over $C$ such that $\sigma(a) = b$;
    \item the completions of $\ACFG$ are given by the $\LL_G$-isomorphism type of $(\overline{\F_p}, G(\overline{\F_p}))$.
  \end{enumerate}

Let $x$ be a tuple from a field extension of $K$ and $H$ be an additive subgroup of the field $\overline{Cx}$. If
   $$\overline{Cx}\cap K = \overline{C}\mbox{ and }H\cap \overline{C} = G(\overline{C})$$
 then, by~\cite[Proposition 1.16]{dE18A}, the type associated to the $\LL_G$-isomorphism class of the pair $(\overline{Cx}, H)$ is consistent in $(K,G)$, i.e. there exists a tuple $a$ from $K$ such that there is a $\LL_G$-isomorphism over $C$ $$f : (\overline{Ca},G(\overline{Ca})) \rightarrow  (\overline{Cx},H)$$
   with $f(a) = x$. 



 \begin{ex}[Empty types]
   Let $(K,G)$ be a $\kappa$-saturated model of $\ACFG$, $C\subseteq K$ such that $\abs{C}<\kappa$ and $x$ a finite tuple algebraically independent over $K$. By previously, the type associated to the pair $(\overline{Cx}, G(\overline{C}))$ is consistent. Hence there is some tuple $a$ from $K$, algebraically independent over $C$ such that $G(\overline{Ca}) = G(\overline{C})$.
   This type is unique if $G(\overline{C})\subseteq C$: let $a$ and $a'$ realise this type, meaning that $G(\overline{Ca}) = G(\overline{Ca'}) = G(\overline{C})$. Then $a \equiv_C a'$. Indeed if $\sigma$ is a field isomorphism over $C$ between $\overline{Ca}$ and $\overline{Ca'}$, then it fixes $G(\overline{C})$ so it is an $\LL_G$-isomorphism. The type is unique in particular if $C$ is algebraically closed. This uniqueness is a special case of the stationarity of the strong independence see Proposition~\ref{lm_propst}.
 \end{ex}

\subsection{Independence relations in $(K,G)$}
We work in a monster model $(K,G)$ of $\ACFG$.

\begin{df}[Weak and strong independence]
Let $A,B,C$ be subsets of $K$. Let $\indi{\ACF}\quad$ be the forking independence in the sense of $\ACF$. Recall the \emph{weak independence relation}:
$$A\indi w _C B \text{ if and only if } A\indi{\ACF}_C \quad B\text{ and } G(\overline{AC}+\overline{BC}) = G(\overline{AC})+G(\overline{BC}),$$
and the \emph{strong independence relation}:
$$A\indi{st} _C B \text{ if and only if } A\indi{\ACF}_C \quad B\text{ and } G(\overline{ABC}) = G(\overline{AC})+G(\overline{BC}).$$
\end{df}

\begin{thm}\label{thm_indw}
  The relation $\indi w$ satisfies \ref{INV}, \ref{CLO}, \ref{SYM}, \ref{EXT}, \ref{MON}, \ref{EX}, \ref{LOC}, \ref{TRA}, \ref{STRFINC} over algebraically closed sets, $\indi a$-\ref{AM} over algebraically closed sets.
\end{thm}
\begin{proof}
  Apart from \ref{TRA} and \ref{LOC}, all properties has been proven in~\cite[Theorem 4.2]{dE18A} and~\cite[Example 5.21]{dE18A}.
\\

\ref{TRA}. Assume that $A \indi w _{CB} D$ and $B\indi w_C D$. We may assume that $A = \overline{ABC},B=\overline{CB}$ and $D=\overline{CD}$. By \ref{MON}, it is sufficient to show that $A\indi w _C D$. We clearly have $A\indi{\ACF}_C\quad D$ by \ref{TRA} of $\indi{\ACF}\quad$. We show that $G(A+D) = G(A)+G(D)$. By $A\indi w _B D$ we have $G(A + \overline{BD}) = G(A)+G(\overline{BD})$. It follows that $G(A+D)$ is included in $(A+D)\cap (G(A)+G(\overline{BD}))$, which, by modularity, is equal to $$G(A)+ (A+D)\cap G(\overline{BD}) = G(A)+ G(A\cap \overline{BD} + D).$$
As $A\indi{\ACF}_B\quad D$, $A\cap \overline{BD} = B$. By $B\indi w _C D$, $G(B+D) = G(B)+G(D)$ hence $$G(A+D) = G(A)+G(B)+G(D) = G(A)+G(D).$$

\ref{LOC}.   We start with a claim.
  \begin{claim}Let $A,B$ be subsets of $(K,G)$ with $B$ subgroup of $(K,+)$, then there exists $C\subseteq B$ with $\abs{C}\leq \abs{A}$ such that $$G(A+B) = G(A+C)+G(B).$$
  \end{claim}
  \begin{proof}[Proof of the claim] For each $a\in A$ define $C(a)$ to be the set of those $b\in B$ such that $a+b\in G$. Take $c(a)\in C(a)$ for each $a$ such that $C(a)$ is nonempty, and set $$C = \set{c(a) \mid a\in A \text{ and }C(a)\neq \emptyset}.$$ Now if $g\in G(A+B)$ then $g=a+b$ with $a\in A$, $b\in B$. We have $C(a)$ nonempty so we can write for $c = c(a)$ $$g =(a+ c)+ (b-c).$$ It follows that $b-c \in G(B)$ hence $g\in G(A+C)+G(B)$. The reverse inclusion is trivial.\end{proof}

  Let $a$ be a finite tuple and $B$ an algebraically closed set. We construct two sequences $(A_i)_{i<\omega}$ and $(D_i)_{i<\omega}$ such that the following holds for all $n<\omega$:\\
  \begin{enumerate}
  \item $A_n\subseteq A_{n+1}\subseteq \overline{Ba}$ and $D_n\subseteq A_{n}$
  \item $G(A_n + B)\subseteq G(A_{n+1})+G(B)$
  \item $A_{n}\indi{\ACF}_{D_n}B$
  \item $\abs{A_n}\leq \aleph_0$
\end{enumerate}
  Using \ref{LOC} for $\indi{\ACF}\quad$ there exists a countable set $D_0\subseteq B$ such that $a \indi{\ACF}\quad_{D_0} B$. We define $A_0 = \overline{a D_0}$. Assume that $D_n$ and $A_n$ has been constructed and that $\abs{A_n}\leq \aleph_0$.
  By the claim there exists $C\subseteq B$ with $\abs{C}\leq \aleph_0$ such that $G(A_n+B) = G(A_n+C)+G(B)$. Using \ref{LOC}\footnote{Here we use a stronger version of \ref{LOC} which holds in any simple (countable) theory (see~\cite[Proposition 5.5]{C11}): for all countable set $A$ and arbitrary set $B$ there exists $B_0\subseteq B$ with $\abs{B_0}\leq \aleph_0$ with $A\ind_{B_0} B$.} of $\indi{\ACF}\quad$ on the set $A_nC$ there exists $D_{n+1}\subseteq B$ with $\abs{D_{n+1}}\leq \aleph_0$ such that $A_n C\indi{\ACF}_{D_{n+1}}\quad B$. We set $A_{n+1} = \overline{A_n C D_{n+1}}$. Note that $A_n+C \subseteq A_{n+1}$ so $G(A_n+B)\subseteq G(A_{n+1})+G(B)$.

Now set $A_\omega = \bigcup_{i<\omega} A_i$ and $D_{\omega} = \bigcup_{i<\omega} D_i$. We have $\abs{A_\omega}\leq \aleph_0$ and $\abs{D_\omega}\leq \aleph_0$. We claim that $$A_\omega \indi{w}_{D_\omega} B.$$
If $u$ is a finite tuple from $A_\omega$, then $u\subseteq A_{n}$ for some $n$, so as $A_{n}\indi{\ACF}_{D_n}\quad B$ we have $u\indi{\ACF}_{D_n}\quad B$. Now as $D_\omega \subseteq B$, we use \ref{BMON} of $\indi{\ACF}\quad$ to conclude that $u\indi{\ACF}_{D_\omega}\quad B$. As this holds for every finite tuple $u$ from $A_\omega$, we conclude that $$A_\omega \indi{\ACF}_{D_\omega}\quad B.$$
It remains to show that $G(A_\omega+B) = G(A_\omega)+G(B)$. If $g\in G(A_\omega+B)$ then there is some $n$ such that $g\in A_n+B$ and so
$$g\in G(A_n+B)\subseteq G(A_{n+1})+G(B)\subseteq G(A_\omega)+G(B).$$
The reverse inclusion being trivial, we conclude that $G(A_\omega + B)= G(A_\omega)+G(B)$, so $A_\omega \indi{w}_{D_\omega} B$.
As $a\subseteq A_\omega$ we conclude by \ref{MON} of $\indi w$.
\end{proof}

\begin{rk} The Kim-Pillay theorem states that if a relation $\ind$ satisfies \ref{INV}, \ref{SYM}, \ref{MON}, \ref{BMON}, \ref{TRA}, \ref{EXT}, \ref{LOC}, $\ind$-\ref{AM} over models and \ref{FIN}\footnote{This property is trivial for $\indi w$ and $\indi{st}$.}, then the theory is simple and this relation is forking independence. From Theorem~\ref{thm_indw} and~\cite[Proposition 5.20]{dE18A}, the weak independence $\indi w$ satisfies all the previous properties except \ref{BMON}. This is similar to the case of $K_{n,m}$-free bipartite graph~\cite[Remark 4.17]{CKr17}. 
\end{rk}

\begin{prop}\label{prop_kim}
  Assume that $C=\overline{C}$. If $a\indi w _C b$, then for all $C$-indiscernible sequence $(b_i)_{i<\omega}$ in $tp(b/C)$ such that $b_i\indi a _C (b_j)_{j<i}$ there exists $a'$ such that $a'b_i \equiv_C ab$ for all $i<\omega$.
In particular, the following are equivalent, for $C$ algebraically closed and $a\indi{\ACF}_C \quad b$.
\begin{enumerate}
\item $a\indi w _C b$;
\item for all $C$-indiscernible sequence $(b_i)_{i<\omega}$ in $tp(b/C)$ such that, $b_i\indi a _C (b_j)_{j<i}$ and $G(\overline{Cb_i}+\overline{Cb_k}) = G(\overline{Cb_i})+G(\overline{Cb_k})$ there exists $a'$ such that $a'b_i \equiv_C ab$ for all $i$;
\item for some $C$-indiscernible sequence $(b_i)_{i<\omega}$ in $tp(b/C)$ such that, $b_i\indi a _C (b_j)_{j<i}$ and $G(\overline{Cb_i}+\overline{Cb_k}) = G(\overline{Cb_i})+G(\overline{Cb_k})$ there exists $a'$ such that $a'b_i \equiv_C ab$ for all $i$.
\end{enumerate}
\end{prop}
\begin{proof}
The first assertion holds because $\indi w$ satisfies $\indi a$-\ref{AM} over algebraically closed sets (Theorem~\ref{thm_indw}). The proof is a classical induction similar to the proof of Lemma~\ref{lm_fork} or~\cite[Proposition 4.11]{CKr17}. 

\textit{(1)} implies \textit{(2)} is a particular case of the first assertion. \textit{(2)} implies \textit{(3)} follows from the fact that such sequence exists, which follows from \ref{EXT} of $\indi w$. We show that \textit{(3)} implies \textit{(1)}. Assume that $a\nindi w _C b$ and let $\Lambda(x,b,c)$ be as in~\cite[Lemma 3.6]{dE18A}. If \textit{(3)} holds, then in particular $\set{\Lambda(x,b_i,c)\mid i<\omega}$ is consistent, for some $(b_i)_{i<\omega}$ such that $b_i\equiv_C b$ and $b_i\indi a_C b_j$. This contradicts~\cite[Lemma 3.6]{dE18A}.
\end{proof}

In particular, we have the following combinatorial characterization of $\indi w$ over algebraically closed sets.

\begin{cor}\label{cor_kim}
  The following are equivalent, for $C$ algebraically closed
  \begin{enumerate}
    \item $a\indi w _C b$;
    \item for all $C$-indiscernible sequence $(b_i)_{i<\omega}$ in $tp(b/C)$ such that, $b_i\indi w _C (b_j)_{j<i}$ there exists $a'$ such that $a'b_i \equiv_C ab$ for all $i$;
    \item for some $C$-indiscernible sequence $(b_i)_{i<\omega}$ in $tp(b/C)$ such that, $b_i\indi w _C (b_j)_{j<i}$ there exists $a'$ such that $a'b_i \equiv_C ab$ for all $i$.
  \end{enumerate}
\end{cor}

\begin{proof}
  \textit{(1)} implies \textit{(2)} follows from Proposition~\ref{prop_kim}, and \textit{(2)} implies \textit{(3)} holds since $\indi w$ satisfies \ref{EXT}. Assume that \textit{(3)} holds for some $a'$ and indiscernible sequence $(b_i)_{i<\omega}$ such that $b_i\indi w _C (b_j)_{j<i}$for all $i<\omega$. In particular, $(b_i)_{i<\omega}$ is a Morley sequence in the sense of $\ACF_p$, and $a'b_i \equiv^{\ACF}_C ab$ for all $i<\omega$. As $\indi{\ACF}\quad$ is forking independence in the sense of $\ACF_p$, we have $a\indi{ACF}_C \quad b$. By Proposition~\ref{prop_kim} we have $a\indi w _C b$.
\end{proof}

\begin{prop}\label{lm_propst}
  The relation $\indi{st}$ satisfies \ref{INV}, \ref{FIN}, \ref{SYM}, \ref{CLO}, \ref{MON}, \ref{BMON}, \ref{TRA}, \ref{EXT}, \ref{STAT} over algebraically closed sets.
\end{prop}
\begin{proof}
  This is~\cite[Lemma 3.10]{dE18A}.
\end{proof}

\begin{rk}
  The property $\ind$-\ref{AM} over models follows from \ref{STAT} over algebraically closed sets, hence from Proposition~\ref{lm_propst}, the strong independence $\indi{st}$ satisfies every property of the Kim-Pillay characterization except \ref{LOC} otherwise, $\ACFG$ would be simple. Example~\ref{ex_monloc} shows directly that \ref{LOC} is not satisfied by $\indi{st}$, nor by any relation stronger than $\indi w$ which satisfies \ref{BMON}. Finally, by Proposition~\ref{lm_propst} and~\cite[Remark 4.7]{dE18A}, $\ACFG$ is mock stable in the sense of Adler~\cite{Ad08}.
\end{rk}

\subsection{Some structural features of $(K,G)$}\label{sec_structfeat} 

Let $P(X)$ be a polynomial in variables $X = X_1,\dots,X_n$ with coefficients in $K$. We say that $P$ is \emph{$\F_p$-flat over $K$} if whenever $u$ is a zero of $P$ in some field extension of $K$, there exists a non trivial $\F_p$-linear combination of $u$ that falls in $K$.

\begin{lm}\label{lm_SAG}
Let $(K,G)$ be an $\aleph_0$-saturated model of $\ACFG$, and $P(X_1,\dots,X_n)$ a polynomial non-$\F_p$-flat over $K$. Then for every $I\subset \set{1,\dots,n}$ there exists a zero $a$ of $P$ in $K$ such that $a_i\in G \iff i\in I$.
\end{lm}
\begin{proof}
Let $I\subset \set{1,\dots,n}$. As $P$ is non-$\F_p$-flat, there exists a zero $t$ of $P$ in an extension of $K$ such that no non nontrivial $\F_p$-combination of $t$ falls in $K$. It follows that $(\overline{K(t)}, G+\vect{t_i \mid i\in I})$ is an $\LL_G$-extension of $(K,G)$. Indeed $(G+\vect{t_i \mid i\in I})\cap K = G$. Furthermore $t_j\in (G+\vect{t_i \mid i\in I})$ if and only if $j\in I$. As $(K,G)$ is existentially closed in $(\overline{K(t)}, G+\vect{t_i \mid i\in I})$, we have that $$(K,G)\models \exists x (P(x)=0 \wedge \bigwedge_{i\in I} x_i\in G \wedge \bigwedge_{j\notin I} x_j\notin G).$$
\end{proof}

\begin{lm}\label{lm_SAGV}
  A polynomial $P$ in $K[X]$ is $\F_p$-flat over $K$ if and only if all its irreducible factors in $K[X]$ are of the form $c(\lambda_1X_1+\dots+\lambda_n X_n -b)$ for some $\lambda_1,\dots,\lambda_n$ in $\F_p\setminus\set{0}$ and $b,c\in K$.
\end{lm}
\begin{proof}
  Assume that $P$ is $\F_p$-flat over $K$. If $\abs{X} = 1$, then $P$ satisfies the conclusion. Assume that $\abs{X}>1$. Let $t_2,\dots,t_n$ be algebraically independent over $K$, and consider $P(X_1,t_2,\dots,t_n)$. This polynomial has zeros in $\overline{K(t_2,\dots,t_n)}$ hence by $\F_p$-flatness each root $u$ satisfies $\lambda_1u + \lambda_2t_2+\dots+\lambda_n t_n =b$ for some non-zero tuple $\lambda_1,\dots,\lambda_n$ from $\F_p$ and $b\in K$. By hypothesis on $t_2,\dots,t_n$ we have that $\lambda_1\neq 0$.
It follows that $X_1-\lambda_1^{-1}(\lambda_2t_2+\dots+\lambda_n t_n -b)$ divides $P(X_1,t_2,\dots,t_n)$ hence $\lambda_1X_1+\dots+\lambda_n X_n -b$ divides $P$, as $K[X_1,t_2,\dots,t_n]\cong K[X]$. If $\lambda_i = 0$ for some $i$, then the tuple $(0,\dots,t,\dots,0)$ with $t$ transcendental over $K$ at the $i$-th coordinate, is a zero of $P$ that contradicts the $\F_p$-flatness. It follows that $P$ is of the desired form. The other direction is trivial.
\end{proof}

\begin{ex}[$\F_p$-flatness might depends on $p$]
  Consider the polynomial $P = X^2+Y^2$, with $b\in K$. Then $P$ is $\F_p$-flat over any algebraically closed field if and only if $-1$ is a square in $\F_p$. From ~\cite[Exercice 1.9.24]{F01}, when $p>2$ this is equivalent to $p\in 4\Z+1$. Using Lemmas~\ref{lm_SAG} and~\ref{lm_SAGV} it follows that whenever $(K,G)\models\ACFG$, $p>2$,
  \begin{itemize}
    \item if $p\notin 4\Z+1$ there exists $g\in G$ and $u\in K\setminus G$ such that $g^2+u^2 = 0$;
    \item if $p\in 4\Z+1$ such couple $(u,g)$ does not exists in $(K,G)$, as every couple of solution to $X^2+Y^2 = 0$ are $\F_p$-linearly dependent.
  \end{itemize}
\end{ex}

For two sets $A$ and $B$ in a field, we denote by $A\cdot B$ the product set $\set{ab\mid a\in A,\ b\in B}$.
\begin{prop}\label{prop_propmod}
  Let $(K,G)$ be a model of $\ACFG$. The following holds:
\begin{enumerate}
\item $K = G\cdot G = G\cdot(K\setminus G) = (K\setminus G)\cdot (K\setminus G)$;
\item $G$ is stably embedded in $K$;
\item For $a\notin \F_p$ and $P\in K[X]\setminus (K+\F_p\cdot X)$, we have $K = G+aG = (K\setminus G)+aG = G+P(G)$.
\end{enumerate}
\end{prop}
\begin{proof}
\textit{(1)} For all $b\in K$ the polynomial $XY-b$ is not $\F_p$-flat by Lemma~\ref{lm_SAGV}, so we conclude using~Lemma~\ref{lm_SAG}.\\
\textit{(2)} From \textit{(1)}, every element in $K$ is product of two elements in $G$, so any $\LL_G$-formula $\phi(x,a_1,\dots,a_n)$ is equivalent to $\phi(x,g_1h_1,\dots,g_nh_n)$ with $g_i,h_i\in G$.\\
\textit{(3)} For all $P\in K[X]\setminus (K+\F_p\cdot X)$, $b\in K$, the polynomial $Y+P(X)-b$ is not $\F_p$-flat, similarly to \textit{(1)}.
\end{proof}

\begin{prop}\label{prop_endodef}
Let $\zeta_1,\dots,\zeta_n$ be $\LLr$-definable endomorphisms of $(K,+)$, $\F_p$-linearly independent. Then
$$K/(\zeta_1^{-1}(G)\cap\dots\cap\zeta_n^{-1}(G)) \cong K/\zeta_1^{-1}(G) \times \dots\times K/\zeta_n^{-1}(G).$$
\end{prop}

\begin{proof}
Using the first isomorphism theorem, it is sufficient to prove that the function $\zeta:K\rightarrow K/\zeta_1^{-1}(G) \times \dots\times K/\zeta_n^{-1}(G)$ defined by $\zeta(u) = (u+\zeta_1^{-1}(G),\dots, u+\zeta_n^{-1}(G))$ is onto. Let $c_1,\dots,c_n\in K$, we want to show that there exists $c\in K$ such that for all $i$
$\zeta_i(c-c_i)\in G$. Let $t$ be a transcendental element over $K$, by model completeness of $\ACF_p$, $\zeta_1,\dots,\zeta_n$ are $\F_p$-linearly independent definable endomorphisms of $(\overline{Kt},+)$. Consider the $\LL_G$-structure $$(\overline{Kt}, G+\vect{\zeta_i(t-c_i)\mid i\leq n}).$$
We have $(G+\vect{\zeta_i(t-c_i)\mid i\leq n})\cap K = G+\vect{\zeta_i(t-c_i)\mid i\leq n}\cap K$. For $\lambda_1,\dots,\lambda_n\in \F_p$, if $\sum_i \lambda_i \zeta_i(t-c_i)\in K$ then $\sum_i \lambda_i \zeta_i(t)\in K$. It is standard that a definable additive endomorphism is of the form $x\mapsto a_1\mathrm{Frob}^{n_1}(x)+\cdots+a_k\mathrm{Frob}^{n_k}(x)$ with $n_i\in \Z$ (see \cite[Chapter 4, Corollary 1.5]{B98} and \cite[Lemma A, VII, 20.3]{Hum98}) hence there is some $k$ such that $t\mapsto \left(\sum_i \lambda_i \zeta_i(t)\right)^{p^k}$ is polynomial. As $t$ is transcendental over $K$, $\left(\sum_i \lambda_i \zeta_i\right)^{p^k} = 0$, so $\sum_i \lambda_i \zeta_i = 0$. As $\zeta_1,\dots,\zeta_n$ are $\F_p$-linearly independent, $\lambda_1=\dots=\lambda_n = 0$. It follows that $(G+\vect{\zeta_i(t-c_i)\mid i\leq n})\cap K = G$, so $(\overline{Kt}, G+\vect{\zeta_i(t-c_i)\mid i\leq n})$ extends $(K,G)$. As $(K,G)$ is existentially closed in $(\overline{Kt}, G+\vect{\zeta_i(t-c_i)\mid i\leq n})$ we have that $(K,G)\models \exists x \bigwedge_i \zeta_i(x-c_i)\in G$, hence $\zeta$ is onto.
\end{proof}

If $\zeta_1,\dots,\zeta_n$ are $\F_p$-linearly independent $\LLr$-definable isomorphisms of $(K,+)$, the previous result can be used to find canonical parameters for the quotient $K/(\zeta_1^{-1}(G)\cap\dots\cap\zeta_n^{-1}(G))$ provided one have canonical parameters for the quotient $K/G$, see Example~\ref{ex_elim}.

\subsection{Models of $\ACFG$ in $\overline{\F_p}$}\label{sub_baire}

From~\cite[Theorem 5.2]{dE18A}, for any quantifier free $\LLr$-formula $\phi(x,y)$, there exists an $\LLr$-formula $\theta_\phi(y)$ such that for $K\models \ACF_p$ sufficiently saturated and $b$ tuple in $K$ such that $K\models \theta_\phi(b)$ if and only if there exists a realisation $a$ of $\phi(x,b)$ which is $\F_p$-linearly independent over $\ol{\F_p(b)}$. By quantifier elimination in $\ACF_p$, the formula $\theta_\phi$ can be choosen quantifier-free.

\begin{lm}\label{lm_Fp}
  If $\F_{p^n}\models \theta_\phi(b)$ then for all $m\geq n$ there exists $k> m$ such that
  $$\F_{p^k}\models \exists x\phi(x,b)\wedge\text{$x$ is $\F_p$-linearly independent over $\F_{p^m}$}.$$
\end{lm}
\begin{proof}
  Assume that $\F_{p^n}\models \theta_\phi(b)$. Then as $\theta_\phi$ is quantifier free, $\ol{\F_p}\models \theta_\phi(b)$. It follows that for some elementary extension $K$ of $\ol{\F_p}$, there is some realisation $a$ of $\phi(x,b)$ which is $\F_p$-linearly independent over $\ol{\F_p}$. In particular for every non trivial polynomial $P(Z,Y)\in \F_p[Z,Y]$ (where $Z$ is a single variable and $Y$ a tuple of variables with $\abs{Y} = \abs{y}$), no nontrivial $\F_p$-linear combination of $a$ is a root of $P(Z,b)$.
  As $\overline{\F_p} \equiv^{\ACF} K$, the following sentence holds in $\oFp$:
$$ \forall y (\theta_\phi(y)\rightarrow (\exists x \phi(x,y) \wedge \mbox{"no nontrivial $\F_p$-linear combination of $x$ is a root of $P(Z,y)$"})).$$
In particular, for the polynomial $X^{p^m}-X$ for some $m$ we have
$$\overline{\F_p} \models \exists x \phi(x,b)\wedge\text{"no non-trivial $\F_p$-linear combination of $x$ falls in $\F_{p^m}$"}.$$
Hence for some $k>m,n$ there exists a tuple $a$ from $\F_{p^k}$ such that $$\overline{\F_{p}}\models \phi(a,b)\wedge\text{"$a$ is $\F_p$-linearly independent over $\F_{p^m}$"}.$$ 
As $\phi(x,y)$ is quantifier-free, we also have that $$\F_{p^k}\models \phi(a,b)\wedge\text{"$a$ is $\F_p$-linearly independent over $\F_{p^m}$"}.$$
\end{proof}

\begin{prop}\label{ACFG_models}
  For any $n\in \N$ and any $G_0$ additive subgroup of $\F_{p^n}$ there exists a subgroup $G$ of $\oFp$ such that $ G\cap \F_{p^n} = G_0$ and $(\overline{\F_p},G)\models\ACFG$.
\end{prop}

\begin{proof}
Start with the following claim.
\begin{claim}
  Let $n\in \N$, let $s\in \N$, let $k_1,\dots, k_s \in \N$ and let $\phi_1(x^1,y^1),\dots,\phi_s(x^s,y^s)$ be quantifier free formulae in $\LLr$. For $i\leq s$, let $B_i = \set{b\in \F_{p^n}^{\abs{y^i}} \mid b\models \theta_{\phi_i}(y)}$. Then there exists $m >n$ such that for all $i\leq s$ and $b\in B_i$ there exists some $\abs{x^i}$-tuples $a^{i,1},\dots,a^{i,k_i}$ (depending on $b$) from $\F_{p^m}$ such that
  \begin{enumerate}
    \item $(a^{i,j}_k)_{i\leq s, j\leq k_i,k\leq \abs{x^{i}}}$ is a $\F_p$-linearly independent tuple over $\F_{p^n}$
    \item $\F_{p^m} \models \phi_i(a^{i,1},b),\dots,\F_{p^m}\models \phi_i(a^{i,k_i},b)$.
  \end{enumerate}
\end{claim}
\begin{proof}[Proof of the Claim]
  We do it step by step, as there are only a finite number of tuples to add. Start with $\phi_1(x^1,y^1)$. Take a first $b\in B_1$. As $\F_{p^n}\models \theta_{\phi_1}(b)$, we use Lemma~\ref{lm_Fp} with $m=n$ to get a first $m_1>n$ such that there exists $a^1\in \F_{p^{m^1}}^{\abs{x_1}}$ such that $\models \phi_1(a^1, b)$ and $a^1$ is $\F_p$-linearly independent over $\F_{p^n}$. Using again Lemma~\ref{lm_Fp} with $m=m_1$ there exists $m_2>m_1$ and a second $a^2\in \F_{p^{m_2}}^{\abs{x^1}}$ such that $\F_{p^{m_2}}\models \phi_1(a^2, b)$ and $a^2$ is $\F_p$-linearly independent over $\F_{p^{m_1}}$. In particular $a^2$ is $\F_p$-linearly independent from $a^1$ over $\F_{p^n}$. So we can construct as many (finitely) solution to $\phi_1(x^1, b)$ as we want which are $\F_p$-linearly independent over $\F_{p^n}$. Once we have enough $\F_p$-linearly independent solutions of $\phi_1(x,b)$, we can do the same trick with another $b'\in B_1$, and add as many (finitely) solution as we want, $\F_p$-linearly independent from one another and from the ones corresponding to $b$, in a finite extension of $\F_{p^{n}}$. Once we have done it for all elements of $B_1$, we do the same with every element $b\in B_2$, continuing to use Lemma~\ref{lm_Fp} to get solutions of $\phi_2(x^2,b)$ $\F_p$-linearly independent from one another and from the previous ones. As every $B_i$ is finite and they are in finite number, we can finish to add $\F_p$-linearly independent solutions of $\phi_i$ in a finite number of steps and the claim is proven.
\end{proof}
From Proposition~\ref{prop_axACFG}, the axioms for $\ACFG$ are given by the following scheme: for all quantifier free $\LLr$-formula $\phi(x,y)$, for all $0\leq  k\leq \abs{x}$ and $0\leq k'\leq \abs{y}$ $$\forall y \left( (\theta_\phi(y)\wedge \vect{y_1,\dots,y_{k'}}\cap G = \set{0})\rightarrow (\exists x \phi(x,y)\wedge \vect{x,y_1,\dots,y_{k'}}\cap G = \vect{x_1,\dots,x_k})\right)$$
with the following convention: $a_1,\dots,a_0 = \emptyset$. 
We will denote the previous sentence by $\Gamma(\phi,k,k')$.
Now we construct by induction a model of $\ACFG$ starting from $(\F_{p^n}, G_0)$. Let $(\phi_i(x^i,y^i))_{i<\omega}$ be an enumeration of all quantifier-free formula in $\LLr$. We construct an increasing sequence $(n_j)_{j<\omega}$ starting with $n_0 = n$ and additive subgroups $G_j$ of $\F_{p^{n_j}}$ such that for all $s<\omega$, for $\phi_1(x^1,y^1),\dots,\phi_s(x^s,y^s)$, for all $1\leq l \leq s$, for all $0\leq k\leq \abs{x^l}$ and $0\leq k'\leq \abs{y^l}$ the following holds for all $\abs{y^l}$-tuples $b$ from $\F_{p^{n_s}}$
\begin{center}
  If $\F_{p^{n_s}}\models \theta_{\phi_l}(b)\wedge \vect{b_1,\dots, b_{k'}}\cap G_s = \set{0}$ then there exists $a^{l,k}$ an $\abs{x^l}$-tuple from $\F_{p^{n_{s+1}}}$ such that $\F_{p^{n_{s+1}}} \models \phi(a^{l,k},b)\wedge \vect{a^{l,k},b_1,\dots,b_{k'}}\cap G_{s+1} = \vect{a^{l,k}_1,\dots,a^{l,k}_k}$.\hfill{($\star$)}
\end{center}
Assume that for some $s<\omega$ we have $n_0,\dots,n_s$ and $G_0\subseteq \F_{p^{n_0}},\dots, G_s\subseteq \F_{p^{n_s}}$ constructed as above.
For every $i\leq s$, we define as above $B_i= \set{b\in \F_{p^{n_s}}^{\abs{y^i}} \mid b\models \theta_{\phi_i}(y)}$, and we apply the claim with $k_i = \abs{x^i}+1$, to get some $n_{s+1}>n_s$. For each $1\leq i \leq s$ and $b\in B_i$ we have $\abs{x^i}+1$ many $\abs{x^i}$-tuples $a^{i,1}(b), \dots, a^{i,k_i}(b)$ from $\F_{p^{n_{s+1}}}$ all $\F_p$-independents over $\F_{p^{n_s}}$ and such that for all $j$, we have $\F_{p^{n_{s+1}}}\models \phi_i(a^{i,j}(b),b)$. Now define $G_{s+1}$ to be

$$G_s\oplus \bigoplus_{1\leq i\leq s}\bigoplus_{b\in B_i}\vect{a_1^{i,2}(b)}\oplus\vect{a_1^{i, 3}(b),a_2^{i,3}(b)}\oplus \dots\oplus \vect{a_1^{i,k_i}(b),\dots,\oplus a_{k_i}^{i,k_i}(b)}.$$
We extend $G_{s}$ by the low triangle of the $(\abs{x_i}+1)\times \abs{x_i}$ matrix $(a^{i,j}_k(b))_{1\leq j\leq k_i,1\leq k\leq \abs{x^i}}$ for each $i<s$ and $b\in B_i$:
$$
\begin{pmatrix}
  a^{i,1}_1 & a^{i,1}_2 & \dots & a^{i,1}_{\abs{x^i}}\\
  \blue{a^{i,2}_1} & a^{i,2}_2 & \dots & a^{i,2}_{\abs{x^i}}\\
  \blue{a^{i,3}_1} & \blue{a^{i,3}_2} & \dots & a^{i,2}_{\abs{x^i}}\\
  \blue{\vdots} &                    & \blue{\ddots}  &    \\
  \blue{a^{i,k_i}_1} & \blue{a^{i,k_i}_2} & \blue{\dots} & \blue{a^{i,k_i}_{\abs{x^i}}}
\end{pmatrix}
.$$

Now we have for each $1\leq i\leq s$ and any $0\leq k\leq \abs{x^i}$ and $0\leq k'\leq \abs{y^i}$, if $b\in B_i$, then there exists $a^{i,k}(b)\in \F_{p^{n_{s+1}}}^{\abs{x^i}}$ such that $  \F_{p^{n_{s+1}}}\models \phi_i(a^{i,k}(b), b)$. By construction if $\vect{b_1,\dots,b_{k'}}\cap G_s = \set{0}$, and by $\F_p$-linear independence of all the $a^{i,k}$, we have $\vect{a^{i,k},b_1,\dots,b_{k'}}\cap G_{s+1} = \vect{a^{i,k}_1,\dots,a^{i,k}_k}$. By induction we construct a familly $(\F_{p^{n_i}},G_i)$ satisfying ($\star$). Now let $$G= \bigcup_{i<\omega} G_i \subseteq \overline{\F_p}.$$ By construction, we have that $(\overline{\F_p},G)$ is a model of $\ACFG$.
\end{proof}

The set of all subgroups of some countable abelian group can be endowed with a topology that is compact, it is called the \emph{Chabauty} topology (see for instance~\cite{deC11}). The Chabauty topology of any countable abelian group is the one of the Cantor space provided that the group is not \emph{minimax}, see~\cite[Proposition A]{deC10}. In the case of the group $(\oFp, +)$, this topology has a very explicit description, in particular, it is the topology of the Cantor space, and is generated by clopens of the form $$\B(H_0,\F_{p^n}) = \set{H\in \Sg(\oFp) \mid  H\cap \F_{p^n} = H_0}$$
for some finite group $H_0\in \Sg(\oFp)$.
Let 
$$\CCC = \set{G\in Sg(\ol{\F_p}) \mid (\ol{\F_p},G)\models \ACFG}.$$
Recall that a set is $G_\delta$ if it is a countable intersection of open sets.
\begin{prop}\label{prop_baire}
  $\CCC$ is a dense $G_\delta$ of $Sg(\ol{\F_p})$.
\end{prop}
\begin{proof}
We first show that it is dense. By Proposition~\ref{ACFG_models}, evey ball of the form $\B(G_0, \F_{p^n}) = \set{H\in \Sg(\oFp)\mid H\cap \F_{p^n} = G_0}$ contains an element of $\CCC$, hence $\CCC$ is dense.

  We show that it is a $G_\delta$. First, from Proposition~\ref{prop_axACFG}, $\ACFG$ is axiomatised by adding to the theory $\ACF_G$ the following $\LL_G$-sentences, for all tuples of variable $x'\subset x$, $y' \subset y$ and $\LLr$-formula $\phi(x,y)$
$$\forall y (\vect{y'}\cap G = \set{0} \wedge \theta_{\phi}(y)) \rightarrow (\exists x \phi(x,y)\wedge \vect{xy'}\cap G = \vect{x'})$$
which is equivalent to 
$$\forall y\exists x \left[\neg \theta_\phi(y))\vee \vect{y'}\cap G \neq \set{0} \vee (\phi(x,y)\wedge \vect{xy'}\cap G = \vect{x'})\right].$$
Let $\phi(x,y)$, $x'\subseteq x$ and $y'\subseteq y$ be given. Let $b$ be a $\abs{y}$-tuple, and consider the set 
$$\OO_b = \bigcup_{a\in \oFp^{\abs{x}},\oFp\models \phi(a,b)} \set{H\mid \vect{b'}\cap H \neq \set{0}}\cup \set{H\mid \vect{ab'}\cap H = \vect{a'}}.$$
The set $\set{H\mid \vect{b'}\cap H \neq \set{0}}$ is equal to $\bigcup_{u\in \vect{b'}\setminus \set{0}}\set{H\mid u\in H}$ which is clearly open. The set $\set{H\mid \vect{ab'}\cap H = \vect{a'}}$ is also open, so $\OO_b$ is open. Now it is an easy checking that 
$$\CCC = \bigcap_{\phi(x,y),  x'\subseteq x, y'\subseteq y}\bigcap_{b\in \oFp^{\abs{y}},\oFp\models \theta_\phi(b)} \OO_b.$$
Hence $\CCC$ is $G_\delta$.
\end{proof}

\begin{rk}[Ultraproduct model of $\ACFG$]\label{rk_ultramod}
From the proof of Proposition~\ref{ACFG_models}, starting from $G_0\subseteq \F_{p^{n_0}}$, there exists a strictly increasing sequence $(n_i)_{i<\omega}$ of integers and an increasing sequence of groups $G_i\subseteq \F_{p^{n_i}}$ satisfying ($\star$).  Let  $\UU$ be a nonprincipal ultrafilter on $\omega$, it does not take long to see that the ultraproduct $ \prod_\UU (\overline{\F_p}, G_i))$ is a model of $\ACFG$, in which the group is pseudo-finite. The construction of the $G_i's$ in the proof of Proposition~\ref{ACFG_models} is rather artificial. Is there more "natural" generic subgroups of $\overline{\F_p}$? Given an arbitrary set $\set{G_i \mid i<\omega}$ of subgroups of $\overline{\F_p}$ and a non principal ultrafilter $\UU$ on $\omega$, how likely is it that $ \prod_\UU (\overline{\F_p}, G_i)) $ is a model of $\ACFG$?
\end{rk}

\begin{rk}[Pseudo finite generic subgroup of a pseudo-finite field]
  Observe that the proof of Lemma~\ref{lm_Fp} gives the following: if $F$ is an infinite \emph{locally finite} field\footnote{A \emph{locally finite} field is a field such that every finitely generated subfield is finite. Equivalently it is embeddable in $\ol{\F_p}$.}, and that for some universal $\LLr$-formula $\phi(x,y)$ there exists an existential formula $\theta_\phi(y)$ such that for all tuple $b$, we have $F\models \theta_\phi(b)$ if and only if there exists a realisation $a$ of $\phi(x,b)$ in an elementary extension of $F$ such that $a$ is $\F_p$-linearly independent over $F$; then for all finite subfields $F_0\subset F_1$ of $F$, if $F_0\models \theta_\phi(b)$ there exists a finite subfield $F_2$ of $F$ and a tuple $a$ from $F_2$ realizing $\phi(x,b)$ which is $\F_p$-linearly independent over $F_1$. By the same method as in the proof of Theorem~\ref{ACFG_models}, we may construct an increasing sequence of finite fields $(F_i)_{i<\omega}$ and finite subgroups $G_i\subseteq F_i$ such that for an enumeration of universal formula $\phi(x,y)$ and existential formula $\theta_{\phi}(y)$, if $(F_i,G_i)$ satisfies the premise of the axiom, then the conclusion is satisfied in $(F_{i+1},G_{i+1})$. Now consider the theory $\Psf_c$ (the theory of pseudo-finite fields with constants for irreducible polynomials, see~\cite[Section 3]{C97}), it is model-complete, hence every formula is equivalent to an existential formula and a universal formula, with some constants. It is then possible to choose constants $c(i)$ in $F_i$ such that $X^n+c_{n-1,n}(i)X^{n-1}+\cdots+c_{0,n}(i)$ is irreducible over $F_i$, for all $n$. Then one can check that a non principal ultraproduct of $(F_i, c(i))_{i<\omega}$ is a model of $\Psf_c$, hence the ultraproduct $\prod_\UU (F_i,c_i,G_i)$
  is a model of the expansion of a pseudo-finite field of characteristic $p$ with a generic subgroup (see~\cite[Example 5.10]{dE18A}). 
\end{rk}

\begin{rk}[Characteristic $0$]
Let $\PP$ be the set of prime numbers and $\UU$ a non-principal ultrafilter on $\PP$. For each $q\in \PP$ let $G_q$ be a subgroup of $\overline{\F_q}$ such that $(\overline{\F_q},G_q)$ is a model of $\ACFG$ (here we mean $\mathrm{ACF_qG}$). Recall that $\C\cong \prod_{q\in \PP} \overline{\F_q}/ \UU$. Consider the ultraproduct
$$(\C,V) \cong \prod_{q\in \PP} (\overline{\F_q},G_q)/ \UU.$$
It is clear that $V$ is a subgroup of $\C$. For each $q\in \PP$, $$Stab_{\overline{\F_q}}(G_q) := \set{a\in \overline{\F_q}\mid a G_q \subseteq G_q}  = \F_q,$$ this follows from Proposition~\ref{prop_propmod} \textit{(3)}. Hence $F = Stab_{\C}(V)$ is a pseudo-finite subfield of $\C$, and $V$ is an $F$-vector space. It follows from~\cite[Proposition 5.2]{dE18A} that $(\C,V)$ is \emph{not} existentially closed in the class of $\LL_G$-structures consisting of a field of characteristic $0$ in which $G$ is an additive subgroup. Nonetheless, some properties such as the ones in Proposition~\ref{prop_propmod} will be satisfied by $(\C,V)$ (replacing $\F_p$ by $F$).
\end{rk}

\section{Imaginaries}\label{sec_im}

Let $(K,G)$ be a saturated model of $\ACFG$. It is easy to see that for all $a\in K\setminus G$, there exists $b\in K\setminus G$ algebraically independent from $a$ over $\F_p$ such that $a-b\in G$ (see Lemma~\ref{lm_fonda}). Let $\alpha = a/G = b/G\in (K,G)^{eq}$. If it exists, a canonical parameter for $\alpha$ in $K$ would be definable over both $a$ and $b$, hence it would be definable over an element of $\overline{\F_p}$. This would give an embedding of $K/G$ into the countable set $\dcl^{eq}(\emptyset)$ which is absurd in a saturated model $(K,G)$ for cardinality reasons.

Let $(K,G)$ be a model of $\ACFG$, there is a canonical projection $$\pi : K \rightarrow K/G.$$

Consider the 2-sorted structure, $(K, K/G)$ with the $\LLr$-structure on $K$, the group structure on $K/G$ (in the language of abelian groups) and the group epimorphism $\pi : K \rightarrow K/G$. We forget about the predicate $G$ as it is $0$-definable in $(K,K/G)$. 
The structure $(K, K/G)$ is bi-interpretable with $(K,G)$. We fix $(K,G)$ and $(K,K/G)$ for the rest of this section.

\noindent In this section, we show that $(K,K/G)$ has weak elimination of imaginaries, hence imaginaries of $(K,G)$ can be weakly eliminated up to the quotient $K/G$.\\

\noindent Some definable imaginaries in $(K,G)$ can be easily eliminated in the structure $(K,K/G)$.
\begin{ex}\label{ex_elim}
  Let $\zeta : K \rightarrow K$ be a $\LLr$-definable group endomorphism. Then in $(K,K/G)^{eq}$, every element in $K/\zeta^{-1}(G)$ is interdefinable with an element in $K/G$. Indeed, for any element $a\in K$ and any automorphism $\sigma$ of $(K,K/G)$, $\sigma(a) - a\in \zeta^{-1}(G)$ if and only if $\sigma$ fixes $\pi(\zeta(a))$, hence $\pi(\zeta(a))$ is a canonical parameter for the class of $a$ modulo $\zeta^{-1}(G)$. 
  
  Let $\zeta_1,\cdots,\zeta_n$ be $\F_p$-linearly independent $\emptyset$-$\LLr$-definable group endomorphisms $K\rightarrow K$. Let $\pi_\zeta : K \rightarrow K/{\zeta_1^{-1}(G)\cap \cdots \cap \zeta_n^{-1}(G)}$ and consider an element $\alpha $ of the sort $K/{\zeta_1^{-1}(G)\cap \cdots \cap \zeta_n^{-1}(G)}$ in $(K,K/G)^{eq}$. From Proposition~\ref{prop_endodef} the natural map $$K/{\zeta_1^{-1}(G)\cap \cdots \cap \zeta_n^{-1}(G)}\rightarrow K/\zeta_1^{-1}(G) \times \cdots \times K/\zeta_n^{-1}(G)$$ is an isomorphism. Let $a$ be such that $\pi_\zeta(a) = \alpha$. For each $1\leq i\leq n$ let $\alpha_i = \pi(\zeta_i^{-1}(a))\in K/G$. Then the tuple $\alpha_1,\cdots,\alpha_n$ is a canonical parameter for $\alpha$. 

\end{ex}
If quotients of the form $K/\zeta^{-1}(G)$ can be fully eliminated, what about quotients of the form $K/\zeta(G)$? In that case the kernel of $\zeta$ is a finite vector space, hence a canonical parameter for $\alpha\in K/\zeta(G)$ is a finite set of the form $\pi(a+\ker(\zeta))$ which is not necessarily eliminable in $(K,K/G)$ as shows Example~\ref{ex_notelim}. We even show in Example~\ref{ex_notelimst} that adding canonical parameters for the sort $K/G$ is not sufficient to eliminate all finite imaginaries of the structure $(K,K/G)$.

In this section, greek letters $\Gamma$, $\alpha$ denote subsets or tuples (which might be infinite) from $K/G$. Any tuple in the structure $(K,K/G)$ will be denoted by $a\gamma$, with $a$ a tuple from $K$, $\gamma$ a tuple from $K/G$. We also extend $\pi$ for (finite or infinite) tuples by $\pi(a) := (\pi(a_i))_{i}$.

\subsection{First steps with imaginaries}\label{subsec_im}

Let $\sigma$ be a field automorphism of $K$. It is clear that the following are equivalent:
\begin{itemize}
\item $\sigma$ is an $\LL^G$-automorphism of $K$;
\item there exists $\tilde{\sigma} :K/G \rightarrow K/G$ such that $ \pi \circ \sigma = \tilde{\sigma} \circ \pi$.
\[
  \begin{tikzcd}
  (K,G) \arrow[r, "\sigma"] \arrow[d, "\pi"]& (K,G) \arrow[d, "\pi"] \\
  K/G \arrow[r, "\tilde{\sigma}"] & K/G
\end{tikzcd}
\]
\end{itemize}
An automorphism of the structure $(K,K/G)$ is a pair $(\sigma,\tilde{\sigma})$ as above. It follows that for $a,b,C$ from $K$, we have 
$$a\equiv_C^{(K,G)} b \iff a\equiv_C^{(K,K/G)}b.$$
In this section, the relation $\equiv$ means having the same type in the structure $(K,K/G)$.

\begin{lm}\label{lm_fonda}
Let $a,b$ be two tuples of the same length from $K$. Let $C,D\subseteq K$ and assume that
\begin{itemize}
\item $\pi(a)$ is an $\F_p$-independent tuple over $\pi(\overline{C})$
\item $\pi(b)$ is an $\F_p$-independent tuple over $\pi(\overline{C})$
\end{itemize}
Then there exists $a'\equiv_C a$ such that $a' \indi{\ACF}_C \quad D$ and $\pi(a') = \pi(b)$. 
\end{lm}

\begin{proof}
Let $x\indi{\ACF}_C\quad K$ such that $x \equiv^{\ACF}_C a$, and $f : \overline{Cx} \rightarrow \overline{Ca}$ a field isomorphism over $C$ sending $x$ to $a$. Let $G_{Cx} = f^{-1}(G(\overline{Ca}))$. Consider now the subgroup of $\overline{CDbx}$ defined by $$H = G_{Cx}+G(\overline{CbD}) + \vect{x_i - b_i \mid i\leq \abs{x}}.$$
We show that the type in the sense of $\ACFG$ defined by the pair $(\overline{CDbx}, H)$ is consistent. 
As $x\indi{\ACF}_C\quad K$ we have $\overline{CDbx}\cap K = \overline{CDb}$. In order to prove that $H\cap \overline{CDb} = G(\overline{CDb})$, it suffices to show that $$\overline{CDb}\cap (G_{Cx}+ \vect{x_i - b_i \mid i\leq \abs{x}})\subseteq G(\overline{C}).$$
Assume that $g_{Cx}+\sum_i\lambda_i(x_i-b_i) \in \overline{CDb}$, where $g_{Cx}\in G_{Cx}$. It follows that $g_{Cx}+\sum_i\lambda_i x_i \in \overline{CDb}$. On the other hand $g_{Cx}+\sum_i\lambda_i x_i \in \overline{Cx}$. As $x\indi{\ACF}_C\quad bD$ we have $\overline{Cx}\cap \overline{CDb} = \overline{C}$ hence $g_{Cx}+\sum_i\lambda_i x_i \in \overline{C}$. 
Apply $\pi\circ f$ to get that $\sum_i\lambda_i\pi(a_i)\in \pi(\overline{C})$ hence by hypothesis $\lambda_i = 0$ for all $i\leq \abs{x}$. It follows that $g_{Cx}\in \overline{C}$ and so $g_{Cx}\in G(\overline{C})$. We have showed that $\overline{CDb}\cap (G_{Cx}+ \vect{x_i - b_i \mid i\leq \abs{x}})\subseteq G(\overline{C})$. The type is consistent by Subsection~\ref{sub_tp}, so realised by say $a'$. As $x\indi{\ACF}_C\quad D$ we have $a'\indi{\ACF}_C\quad D$.
In order to show that $a'\equiv_C a$ we have to check that $H\cap \overline{Cx} = G_{Cx}$, this is similar to the argument above, using this time that $\pi(b)$ is $\F_p$-independent over $\pi(\overline{C})$. We have $a_i'-b_i\in G$ hence $\pi(a_i') = \pi(b_i)$, for all $i\leq \abs{x}$.
\end{proof}

\begin{lm}[Minimal representative]\label{lm_min}
Let $a$, $C$ be in $K$ such that $\pi(a)$ is an $\F_p$-independent tuple over $\pi(\overline{C})$. Then there exists $a'$ of same length as $a$, algebraically independent over $Cb$ such that
\begin{itemize}
\item $\pi(a)=\pi(a')$
\item $\pi(\overline{Ca'}) = \vect{\pi(\overline{C})\pi(a)}$
\item $a'\indi{\ACF}_C\quad b$.
\end{itemize} 
\end{lm}

\begin{proof}
It is again a type to realize. Consider $x$ of same length as $a$ and algebraically independent over $Cba$. Let $V$ be a $\F_p$-vector space complement to $\overline{C}\oplus \vect{x}$ in $\overline{Cx}$ and set 
$$H = G(\overline{Cab})+\vect{x-a}+V.$$
We check that the pair $(\overline{Cabx}, H)$ defines a consistent type over $Cab$. First $H\cap \overline{Cab} = G(\overline{Cab})+ (\vect{x-a}+V)\cap \overline{Cab}$. For $v\in V$, if $\sum_i \lambda_i(x_i-a_i) +v\in \overline{Cab}$ then $\sum_i \lambda_i x_i + v\in \overline{Cab}$. As $\overline{Cab}\cap \overline{Cx} = \overline C$, $\sum_i \lambda_i x_i + v\in \overline{C}$ hence $v =0$ and, as $x$ is $\F_p$-independent over $\overline{C}$, $\lambda_i = 0$ for all $i\leq \abs x$. The type is consistent by Subsection~\ref{sub_tp}. We show that $H\cap \overline{Cx} = G(\overline C )+V$. First $H\cap \overline{Cx} = V+\overline{Cx}\cap (G(\overline{Cab})+\vect{x-a})$. Let $g+ \sum_i \lambda_i (x_i-a_i) \in (G(\overline{Cab})+\vect{x-a})\cap \overline{Cx}$, then $g+\sum_i \lambda_i a_i \in \overline{Cab}\cap \overline{Cx} = \overline C$ and so applying $\pi$ gives $\sum_i \lambda_i \pi(a_i) \in \pi(\overline{C})$ hence $\lambda_i = 0$ for all $i\leq \abs{x}$. It follows that $\overline{Cx}\cap (G(\overline{Cab})+\vect{x-a})=G(\overline{C})$ hence $H\cap \overline{Cx} = G(\overline C )+V$. Assume that $a'$ realises this type, it is clear that $\pi(a)=\pi(a')$ and $a'\indi{\ACF}_C\quad b$. By construction there exists $V'\subseteq \overline{Ca'}$ such that $\overline{Ca'} = \overline C \oplus \vect{a} \oplus V'$ and $G(\overline{Ca'}) = G(\overline{C})\oplus V'$, so it follows that $\pi(\overline{Ca'}) = \pi(C)\oplus \vect{\pi(a')}$.\end{proof}

In particular if $\alpha$ is an $\F_p$-independent tuple over $\pi(\overline{C})$ then there exists some algebraically independent tuple $a$ over $C$ such that $\pi(a) = \alpha$ and $\pi(\overline{Ca}) = \vect{\pi(\overline{C})\alpha}$. We call such a tuple a \emph{minimal representative} of $\alpha$ over $C$. Lemma~\ref{lm_min} states that minimal representatives always exists and that they can be taken independent in the sense of fields from any parameters.

\begin{cor}\label{cor_typesinquo}
Let $\alpha$ and $\beta$ be tuples in $K/G$ of the same length, $\gamma$ tuple from $K/G$ and $C\subseteq K$. If $\alpha$ and $\beta$ are $\F_p$-independent tuples over $\vect{\pi(\overline{C})\gamma}$ then $\alpha \equiv_{C\gamma} \beta$.
\end{cor}

\begin{proof}
  We may assume that $\gamma$ is linearly independent over $\pi(\overline{C})$ and let $r_\gamma $ be a minimal representative of $\gamma$ over $C$. Let $a$ and $b$ be representatives of $\alpha$ and $\beta $ over $\ol{Cr_\gamma}$. Using Lemma~\ref{lm_fonda}, there exists $a' \equiv_{Cr_\gamma} a$ such that $\pi(a') = \pi(b) = \beta$. Let $\sigma$ be an automorphism of $(K,K/G)$ over $Cr_\gamma$ sending $a$ on $a'$. It is clear that $\sigma$ fixes $\gamma$ and sends $\alpha$ to $\beta$ hence $\alpha \equiv_{C\gamma} \beta$.
\end{proof}

\begin{rk}\label{rk_indvect}
  A consequence of Corollary~\ref{cor_typesinquo} is that the induced structure on $K/G$ is the one of a pure $\F_p$-vector space.
\end{rk}
We will describe the algebraic closure $\acl$ in the structure $(K,K/G)$. Surjectivity of $\pi : K\rightarrow K/G$ implies that every formula in the language of $(K,K/G)$ without parameters and with free variables in the home sort $K$ is equivalent to an $\LL^G$-formula. In particular $\acl(C)\cap K = \overline{C}$ for all $C\subseteq K$.

\begin{cor}\label{cor_acleq}
Let $C\subseteq K$ and $\gamma \subseteq K/G$, then 
\begin{itemize}
\item $\acl(C\gamma) \cap K = \overline{C}$
\item $\acl(C\gamma)\cap K/G = \vect{\pi(\overline{C})\gamma}$.
\end{itemize}
\end{cor}

\begin{proof}
For the first assertion, we may assume that $\gamma$ is an $\F_p$-independent tuple over $\pi(\overline{C})$. Let $u$ be in $\acl(C\gamma)\cap K$ witnessed by an algebraic formula $\phi(x,c,\gamma)$ with $c\in C$. 
Using twice Lemma~\ref{lm_min}, let $r_\gamma$ be a minimal representative of $\gamma$ over $C$, and $r_\gamma '$ a minimal representative of $\gamma$ over $C$ such that $r_\gamma ' \indi{\ACF}_C \quad  r_\gamma$. As $u$ satisfies $\phi(x,c,\pi(r_\gamma))$ and $\phi(x,c,\pi(r_\gamma '))$, $u$ belongs to $\overline{Cr_\gamma}\cap \overline{Cr_\gamma '} = \overline{C}$ (note that we don't use the minimality here). The reverse inclusion being trivial, it follows that $\acl(C\gamma) \cap K = \overline{C}$.

For the second assertion, assume that $\alpha \notin \vect{\pi(\overline{C})\gamma}$. By Corollary~\ref{cor_typesinquo}, any element in $K/G\setminus \vect{\pi(\overline{C})\gamma}$ has the same type as $\alpha$ over $C\gamma$ hence $\alpha\notin \acl(C\gamma)$. The reverse inclusion being trivial, it follows that $\acl(C\gamma)\cap K/G = \vect{\pi(\overline{C})\gamma}$.
\end{proof}

\subsection{Independence in $(K,K/G)$}\label{subsec_dual}

Recall the weak independence in $(K,G)$:
$$a\indi{w}_C b \iff a\indi{\ACF}_C \quad \  b \text{ and } G(\overline{Ca} + \overline{Cb}) = G(\overline{Ca})+G(\overline{Cb})$$
It is an easy checking that under the assumption that $\overline{Ca}\cap \overline{Cb} = \overline{C}$ the following two assertions are equivalent:
\begin{itemize}
\item $G(\overline{Ca} + \overline{Cb}) = G(\overline{Ca})+G(\overline{Cb})$
\item $\pi(\overline{Ca})\cap \pi(\overline{Cb}) = \pi(\overline{C})$
\end{itemize}

We define the following relation in $(K,K/G)$:
$$a\alpha \indi{w}_{C\gamma} b\beta \iff a\indi{\ACF}_C \quad \ b \text{ and } \vect{\pi(\overline{Ca})\alpha\gamma}\cap \vect{\pi(\overline{Cb})\beta\gamma} = \vect{\pi(\overline{C})\gamma}$$
It is the right candidate for Kim-independence in $(K,K/G)$. We study only the restriction of this relation to sets $a\alpha,b\beta,C\gamma$ with $\alpha\beta\gamma\subseteq \pi(\overline{Ca})\cap\pi(\overline{Cb})$. This restriction can be described only in terms of the structure $(K,G)$ as we will see now. \\

An infinite tuple $\lambda$ of elements of $\F_p$ is \emph{almost trivial} if $\lambda_i = 0$ for cofinitely many $i$'s. If $\gamma$ is an infinite tuple, an element $u\in \vect{\gamma}$ is an almost trivial linear combination of $\gamma_i$'s, i.e. there exists $\lambda$ almost trivial such that $u=\sum_i \lambda_i \gamma_i$. Given two tuples $a$ and $b$, the tuple consisting of the coordinates $a_i-b_i$ is denoted by $a-b$.

\begin{lm}\label{lm_ind*}
  Let $a,b$ be tuples such that $\gamma $ is a (finite or infinite) tuple from $\pi(\overline{a})\cap \pi(\overline{b})$. Assume that $\overline{a}\cap \overline{b} = \overline{C}$, then the following are equivalent:
\begin{enumerate}
\item $\pi(\overline{a})\cap \pi(\overline{b}) = \vect{\pi(\overline{C})\gamma}$
\item $G(\overline{a}+\overline{b}) = G(\overline{a})+G(\overline{b}) + \vect{r^a - r^b}$ for some (all) representatives $r^a$, $r^b$ of $\gamma$ in $\overline{a}$ and $\overline{b}$ respectively. 
\end{enumerate}
\end{lm}

\begin{proof}
  \textit{(1)} implies \textit{(2)}. Let $u_a\in \overline{a}$ and $u_b\in \overline{b}$ such that $u_a-u_b \in G$. Then $\pi(u_a)=\pi(u_b) \in \pi(\overline{C})+\vect{\gamma}$ so there exists $u_c\in \overline{C}$ and $\lambda\in \F_p^{\abs{\gamma}}$ such that for some (any) representatives $r^a$ and $r^b$ of $\gamma$ in $\overline{a}$ and $\overline{b}$ respectively, there exists $g_a\in G(\overline{a})$, $g_b\in G(\overline{b})$ and an almost trivial sequence $\lambda\in \F_p^{\abs{\gamma}}$ with 
\begin{align*}
u_a = g_a + u_c +\sum_i \lambda_i r^a_i&\\
u_b = g_b + u_c +\sum_i \lambda_i r^b_i.&
\end{align*}
It follows that $u_a-u_b \in G(\overline{a})+G(\overline{b})+\vect{r^a - r^b}$. \\
\textit{(2)} implies \textit{(1)}. If $u_a\in \overline{a}$ and $u_b\in \overline{b}$ are such that $\pi(u_a) = \pi(u_b)$, then $u_a-u_b \in G(\overline{a}+\overline{b})$ hence $u_a-u_b = g_a+g_b + \sum_i \lambda_i (r^a_i - r^b_i)$ (for an almost trivial sequence $\lambda\in \F_p^{\abs{\gamma}}$). It follows that $u_a-g_a - \sum_i \lambda_i r^a_i\in \overline{a}\cap \overline{b} = \overline{c}$, so $\pi(u_a) \in \pi(\overline{C}) +\vect{\gamma}$. \end{proof}

\begin{lm}[Maximal representative]\label{lm_max}
  Let $\gamma$ be a tuple $\F_p$-independent over $\pi(\overline{C})$ and $d$ a tuple from $K$ such that $\pi(d)=\gamma$. Then there exists $(K',G')\succ (K,G)$ and a tuple $r_\gamma$ of length $\abs{\gamma}$ in $K'$, algebraically independent over $K$ such that 
$$G(\overline{Kr_\gamma}) = G(K)+\vect{r_\gamma-d}.$$
Furthermore the following hold for all tuples $a,b$ from $K$ containing $C$ such that $\gamma\in \pi(\overline{a})\cap \pi(\overline{b})$:
\begin{enumerate}
\item if $C = \overline{C}$ then $a\equiv_{C\gamma} b$ if and only if $a\equiv_{\overline{Cr_\gamma}} b$;
\item  $a\indi w _{C\gamma} b$ if and only if $a\indi w _{Cr_\gamma} b$.
\end{enumerate}
\end{lm}

\begin{proof}
  Let $x$ be an algebraically independent tuple over $K$ of size $\abs{d}$, and define $H$ on $K(x)$ to be $G(K) + \vect{x-d}$. It is easy to see that $(\overline{K(x)}, H)$ defines a consistent type over $K$ so let $r_\gamma$ be a realization of this type in an elementary extension $(K',G')$ of $(K,G)$. We may assume that $(K',G')$ is $\kappa$-saturated and $\kappa$-homogeneous for some big enough $\kappa$.

\begin{claim} if $C = \overline{C}$ and $r_\gamma' \equiv_{C\gamma} r_\gamma$ with $r_\gamma'\indi{\ACF}_C\quad  b$ and $G(\overline{Cbr_\gamma '})= G(\overline{b})+\vect{r_\gamma' - r^b}$ for some $r^b\in \pi^{-1}(\gamma)\cap \overline{b}^{\abs{\gamma}}$, then any $\LL^G$-isomorphism over $C\gamma$ that sends an enumeration $R_\gamma'$ of $\overline{Cr_\gamma'}$ to an enumeration $R_\gamma$ of $\overline{Cr_\gamma}$ (and sends $r_\gamma'$ to $r_\gamma$) extends to an $\LL^G$-isomorphism between $\overline{R_\gamma' b}$ and $\overline{R_\gamma b}$ which fixes $b$.
\end{claim}

\begin{proof}[Proof of the Claim]. Let $\sigma$ be an automorphism of $(K',K'/G')$ over $C\gamma$ sending $r_\gamma'$ to $r_\gamma$. Then it sends any enumeration $R_\gamma'$ of $\overline{Cr_\gamma'}$ to an enumeration $R_\gamma$ of $\overline{Cr_\gamma}$. We may assume that $b=\overline{b}$. By stationarity of the type $tp^{\ACF}(b/C)$, the field isomorphism $\sigma\upharpoonright\overline{CR_\gamma'}$ extends to $\tilde{\sigma} : \overline{bR_\gamma'}\rightarrow \overline{bR_\gamma}$ with $\tilde{\sigma}$ fixing $b$. 
We show that $\tilde{\sigma}$ is an $\LL^G$-isomorphism. 
First observe that since $G(\overline{Kr_\gamma})= G(K)+\vect{r_\gamma-r^b}$ then $G(\overline{br_\gamma})= G(b)+\vect{r_\gamma-r^b}$. As $\tilde{\sigma}$ fixes $b$ and sends $r_\gamma'$ to $r_\gamma$ it is clear that $\tilde{\sigma}$ send $G(\overline{br_\gamma'})$ to $G(\overline{br_\gamma})$ so $\tilde{\sigma}$ is an $\LL^G$-isomorphism. 
Now this isomorphism extends to an automorphism of $(K',G')$ and an automorphism of $(K',K'/G')$ that fixes $\gamma$ as it send $r_\gamma'$ to $r_\gamma$. 
\end{proof}

\textit{(1)}. Assume that $a\equiv_{C\gamma} b$ and let $\sigma$ be an automorphism of $(K',K'/G')$ over $C\gamma$ sending $a$ on $b$. As before, we have that $G(\overline{ar_\gamma })= G(\overline{a})+\vect{r_\gamma - r^a}$ and $G(\overline{br_\gamma })= G(\overline{b})+\vect{r_\gamma - r^b}$, for some (any) representatives $r^a, r^b$ of $\gamma$ in $\overline{a}$, $\overline{b}$ respectively. Let $R_\gamma$ be an enumeration of $\overline{Cr_\gamma}$ and $R_\gamma ' = \sigma(R_\gamma)$, $r_\gamma' = \sigma(r_\gamma)$. As $r_\gamma \indi{\ACF}_C \quad a$, we have $r_\gamma' \indi{\ACF}_C \quad b$. Furthermore $G(\overline{ar_\gamma })= G(\overline{a})+\vect{r_\gamma - r^a}$ and $aR_\gamma \equiv_{C\gamma} bR_\gamma'$, then $G(\overline{Cbr_\gamma '})= G(\overline{b})+\vect{r_\gamma' - r^b}$. By the claim, $\sigma^{-1}\upharpoonright \overline{Cr_\gamma'}$ extends $\overline{Cr_\gamma'b}$ with the identity on $b$ hence $R_\gamma \equiv_{Cb\gamma} R_\gamma'$. It follows that $aR_\gamma \equiv_{C\gamma} bR_\gamma$. The other direction is trivial.

\textit{(2)}. From left to right. It is clear that $a\indi{\ACF}_{Cr_\gamma} \quad b$. We want to show that $G(\overline{ar_\gamma }+\overline{br_\gamma}) = G(\overline{ar_\gamma})+G(\overline{br_\gamma})$. Observe that $G(\overline{abr_\gamma}) = G(\overline{ab})+\vect{r^a-r_\gamma}$ for any tuple $r^a$ from $\overline{a}$ with $\pi(r^a)$. 
Let $u\in \overline{ar_\gamma}$ and $v\in \overline{br_\gamma}$. If $u+v\in G$ there exists $g_{ab}\in G(\overline{ab})$ and $\lambda\in \F_p^{\abs{x}}$ such that $u+v = g_{ab} + \sum_i\lambda_i(r^a_i - r_{\gamma i})$ for an almost trivial tuple $\lambda$. 
It follows that $g_{ab} \in (\overline{ar_\gamma} + \overline{br_\gamma})\cap \overline{ab} = \overline{a}+\overline{b}$, using ~\cite[Lemma 5.16.]{dE18A} with $T=\ACF$. As $a\indi w _{C\gamma} b$ and using Lemma~\ref{lm_ind*}, we have that $G(\overline{a}+\overline{b}) = G(\overline{a})+G(\overline{b})+\vect{r^a - r^b}$. We deduce that $g_{ab} = g_a + g_b + \sum_i \mu_i(r^a_i - r^b_i)$, for an almost trivial tuple $\mu$. For all $i$, $r^a_i-r_{\gamma i}\in G(\overline{a r_\gamma})$ and $r_{\gamma i}-r^b_i \in G(\overline{br_\gamma})$ hence $g_{ab}= g_a + g_b + \sum_i \mu_i(r^a_i - r_{\gamma i}) + \sum_i \mu_i(r_{\gamma i} - r^b_i)\in G(\overline{ar_\gamma})+G(\overline{br_\gamma})$. It follows that $u+v \in G(\overline{ar_\gamma})+G(\overline{br_\gamma})$. The other inclusion being trivial we have $G(\overline{ar_\gamma }+\overline{br_\gamma}) = G(\overline{ar_\gamma})+G(\overline{br_\gamma})$. 

From right to left. First, $r_\gamma\indi{\ACF}_C\quad b$ hence by \ref{TRA} and \ref{MON} $a\indi{\ACF}_C\quad b$. By hypothesis, $G(\overline{ar_\gamma }+\overline{br_\gamma}) = G(\overline{ar_\gamma})+G(\overline{br_\gamma})$. Furthermore $G(\overline{ar_\gamma}) = G(\overline a)+\vect{r_\gamma - r^a}$ and $G(\overline{br_\gamma}) = G(\overline b)+\vect{r_\gamma - r^b}$. It is easy to see that $$(G(\overline a)+ G(\overline b)+\vect{r_\gamma - r^a}+\vect{r_\gamma - r^b})\cap (\overline{a}+\overline{b}) = G(\overline a)+ G(\overline b)+\vect{r^b - r^a}.$$
It follows that $a\indi w _{C\gamma} b$.
\end{proof}

  \begin{rk}
  Let $\indi{ST}\ $ be the following relation, defined for $\gamma\in \pi(\overline{Ca})\cap \pi(\overline{Cb})$:
  \begin{align*}
  a \indi{ST}_{C\gamma} \ \ b \iff &a\indi{\ACF}_C \quad b\text{ and }G(\overline{Cab}) = G(\overline{Ca})+G(\overline{Cb})+\vect{r^a_\gamma - r^b_\gamma} \\
                                    &\text{ for some (any) representatives $r^a_\gamma$, $r^b_\gamma$ of $\gamma$ in $\overline{Ca}$, $\overline{Cb}$ respectively}.
  \end{align*}
  A \emph{maximal representative of $\gamma$ over $C$ with respect to $b$} is a representative $r_\gamma$ such that $r_\gamma \indi{ST}_{C\gamma} \ \ b$. The previous result implies that this relation satisfies \ref{EXT} and \ref{STAT} over algebraically closed sets. This relation clearly extends the strong independence in $(K,G)$.
  \end{rk}
  
  \begin{thm}\label{thm_indquo}
  The relation $\indi w$ satisfies the following properties.
  \begin{enumerate}
  \item (Full Existence) Let $a,b,C=\overline{C}$ in $K$ and $\gamma\in K/G$ such that $\gamma\in \pi(\overline{Ca})\cap \pi(\overline{Cb})$ and $\gamma$ $\F_p$-independent over $\pi(\overline{C})$. Then there exists $a'\equiv_{C\gamma} a$ such that $a'\indi w _{C\gamma} b$.
  \item (Transitivity) If $a\alpha\indi{w}_{C\gamma} b\beta$ and $a\alpha \indi w _{Cb\gamma\beta} d\delta$ then $a\alpha\indi w_{C\gamma} bd\beta\delta$
  \item (Independence theorem) Let $c_1,c_2,a,b,C=\overline C$ in $K$ and $\gamma\in K/G$ such that $\gamma\in \pi(\overline{Ca})\cap \pi(\overline{Cb})\cap \pi(\overline{Cc_1})\cap \pi(\overline{Cc_2})$ and $\gamma$ $\F_p$-independent over $\pi(\overline{C})$. 
  
  If $c_1\equiv_{C\gamma} c_2$ and $c_1\indi w _{C\gamma} a, c_2\indi w _{C\gamma} b, a\indi{\ACF}_C \quad b$, then there exists $c$ such that $c\equiv_{Ca\gamma} c_1$, $c \equiv_{Cb\gamma} c_2$ and $c\indi{w}_{C\gamma} a,b$. 
  \end{enumerate}
  \end{thm}

  \begin{proof}
  \textit{Transitivity} is just checking from the definition of $\indi w$. For \textit{Full Existence}, assume the hypothesis and let $r_\gamma$ be a maximal representative as in Lemma~\ref{lm_max}. By \ref{EXT} of $\indi w$ in $(K,G)$ there exists $a'\equiv_{Cr_\gamma} a$ such that $a' \indi{w}_{Cr_\gamma} b$. Using again Lemma~\ref{lm_max}, $a'\equiv_{C\gamma} a$ and $a'\indi w _{C\gamma} b$. 
  For \textit{Independence theorem}, we use the same strategy. Assume the hypothesis and let $r_\gamma$ be a maximal representative of $\gamma$ as in Lemma~\ref{lm_max}. From Lemma~\ref{lm_max}, we have that  $c_1\equiv_{\overline{Cr_\gamma}} c_2$ and $c_1\indi w _{\overline{Cr_\gamma}} a$, $c_2\indi w _{\overline{Cr_\gamma}} b$, $a\indi{\ACF}_{\overline{Cr_\gamma}} b$. As $\indi w$ in $(K,G)$ satisfies $\indi{a}$-\ref{AM} over algebraically closed sets there exists $c$ such that $c\equiv_{\overline{Cr_\gamma}a} c_1$, $c \equiv_{\overline{Cr_\gamma}b} c_2$ and $c\indi{w}_{\overline{Cr_\gamma}} a,b$. It follows that $c\equiv_{Ca\gamma } c_1$, $c \equiv_{Cb\gamma } c_2$, and by Lemma~\ref{lm_max}, $c\indi{w}_{C \gamma} a,b$.
  \end{proof}
  
  \begin{rk}
  Notice that $\indi w$ satisfies $\indi a$-amalgamation over algebraically closed fields in $(K,G)$. In Theorem~\ref{thm_indquo}, we can weaken the hypothesis $a\indi{\ACF}_C \quad b$ to $a\indi a _C b$ because if $a\indi{a}_C b $ and $r\indi{\ACF}_C \quad ab$, then $a\indi a_{Cr} b$ (this result is contained in the proof of Lemma~\ref{mixed_tran}).
  \end{rk}

  \subsection{Weak elimination of imaginaries in $(K,K/G)$}
  
  The following Lemma is a rewriting of the classical argument for the proof of elimination of imaginaries that appears for instance in~\cite{CP98} and~\cite{KrR18}. It is similar to~\cite[Proposition 4.25]{CKr17}, the only difference being that in our case, $\ind$ is defined only on some subsets, and the base set might contain imaginaries, but the proof is the same.

  \begin{lm}\label{lm_critim}
  Let $\MM$ be a $\kappa$-homogeneous and $\kappa$-saturated structure. Let $E\subset \MM^{eq}$. Assume that there exists a binary relation $\ind_E$ on some tuples from $\MM$ such that
  \begin{itemize}
  \item (Invariance) If $a\ind_E b$ and $ab\equiv_E a'b'$ then $a'\ind_E b'$
  \item (Extension) If $a\ind_E b$ and $d$ tuple from $\MM$ then there exists $a'\equiv_{Eb}a$ and $a'\ind_E bd$
  \item (Independent consistency) If $a_1\ind_E a_2$, $b\ind_E a_2$ and $a_2\equiv_E b$, then there exists $a$ such that $a\equiv_{Ea_1} a_2$, $a \equiv_{Ea_2} b$. 
  \end{itemize}
  Let $e\in\MM^{eq}$. If there exists a $0$-definable function $f$ in $\MM^{eq}$ and $a_1,a_2$ in $\MM$ such that $f(a_1) = f(a_2) = e$ and $a_1\ind_E a_2$ then $e\in \dcl^{eq}(E)$.
  \end{lm}
  
  \begin{proof}
  If $e$ is not in $\dcl^{eq}(E)$, then there exists $e'\neq e$ such that $e'\equiv_E e$. Let $\sigma$ be an automorphism of $\MM^{eq}$ over $E$ sending $e$ on $e'$. Let $b_1b_2 = \sigma(a_1a_2)$. By Invariance, $b_1\ind_E b_2$ and $f(b_1)=f(b_2) = e'$. By Extension there exists $b\equiv_{Eb_1}b_2$ such that $b\ind_E a_2$. By Independent Consistency, there exists $a$ such that $a\equiv_{Ea_1} a_2$, $a \equiv_{Ea_2} b$. From $a\equiv_{Ea_1} a_2$ follows that $f(a) = f(a_1) = e$ and from $a \equiv_{Ea_2} b$ follows that $f(a)\neq e$, a contradiction.
  \end{proof}
  
  \begin{rk}
  Recall that Extension follows from Full Existence, Symmetry and Transitivity. Independent consistency is a consequence of the independence theorem. It follows from Theorem~\ref{thm_indquo} that for all $C=\overline{C}$ and $\gamma$ $\F_p$-independent over $\pi(\overline{C})$, the restriction of $\indi{w}_{C\gamma}$ to tuples $a$ such that $\gamma\in \pi(\overline{Ca})$ satisfies the hypothesis of the previous Lemma.
  \end{rk}
  
  The following classical fact follows from a group theoretic Lemma due to P.M. Neumann~\cite{Neu76}. It appears first in~\cite[Lemma 1.4]{EvaHru93}.
 \begin{fact}\label{neumann}
  Let $\MM$ be a saturated model, $X$ a $0$-definable set, $e\in \MM$, $E = \acl(e)\cap X$ and a tuple $a$ from $X$. Then there is a tuple $b$ from $X$ such that 
  $$a\equiv_{Ee} b \mbox{ and } \acl(Ea)\cap \acl(Eb)\cap X = E.$$
\end{fact}

\begin{thm}\label{thm_wei}
Let $e\in (K,G)^{eq}$ then there exists a tuple $c\gamma$ from $(K,K/G)$ such that $c\gamma\in \acl^{eq}(e)$ and $e\in \dcl^{eq}(c\gamma)$. It follows that $(K,K/G)$ has weak elimination of imaginaries.
\end{thm}

\begin{proof}
We work in $(K,G)^{eq}$, seeing $(K,K/G)$ as a $0$-definable subset. Suppose that $e$ is an imaginary element, there is a tuple $a$ from $K$ and a $0$-definable function $f$ such that $e = f(a)$. We set $C\vect{\pi(C)\gamma} = \acl^{eq}(e)\cap (K,K/G)$. We may assume that $\gamma$ is $\F_p$-linearly independent over $\pi(C)$. 
As $\gamma\subseteq \acl^{eq}(e)\cap K/G \subseteq \acl^{eq}(a)\cap K/G$ we have that $\acl^{eq}(Ca\gamma)\cap (K,K/G) = \overline{Ca}\pi(\overline{Ca})$ and $\gamma\subseteq \pi(\overline{Ca})$. By Fact~\ref{neumann} there exists $b\equiv_{C\gamma e} a$ such that $$\acl^{eq}(Ca\gamma)\cap \acl^{eq}(Cb\gamma) \cap (K,K/G) = C\vect{\pi(C)\gamma}.$$
Again, $\acl^{eq}(Cb\gamma)\cap (K,K/G) = \overline{Cb}\pi(\overline{Cb})$ and $\gamma\subseteq \pi(\overline{Cb})$. Furthermore $f(b) = e$ and $$(\overline{Ca}\pi(\overline{Ca}))\cap (\overline{Cb}\pi(\overline{Cb})) = C\vect{\pi(C)\gamma}.$$

\noindent We construct a sequence $(a_i)_{i<\omega}$ such that $$a_{n+1} \indi w _{C\gamma a_n} a_1,\dots , a_{n-1} \mbox{ and }a_n a_{n+1}\equiv_{C\gamma} ab.$$

Start by $a_1 = a$ and $a_2 = b$. Assume that $a_1,\dots,a_n$ has already been constructed. We have that $a_{n-1} \equiv_{C\gamma} a_n$ so let $\sigma$ be a $c\gamma$-automorphism of the monster such that $\sigma(a_{n-1}) = a_n.$
By Full Existence (Theorem~\ref{thm_indquo}) there exists $a_{n+1} \equiv_{Ca_n\gamma } \sigma(a_n)$ such that $a_{n+1}\indi w _{Ca_n\gamma } a_1,\dots,a_{n-1}.$ It follows that $$a_n a_{n+1} \equiv_{C\gamma} a_n\sigma(a_n)\equiv_{C\gamma} a_{n-1}a_n.$$ 
Let $(a_i)_{i<\omega}$ be such a sequence.
In particular the following holds for all $i<j<k$ 
$$a_k \indi{\ACF}_{Ca_j}\quad  a_i \text{, $\ol{Ca_i}\cap \ol{Ca_j} = \ol{C}$} \text{ and } \pi(\overline{Ca_i})\cap \pi(\overline{Ca_j}) = \vect{\pi(C)\gamma}.$$
By Ramsey and compactness we may assume that $(a_i)_{i<\omega}$ is indiscernible over $C\gamma$. As the three properties above holds for the whole sequence, it is in the Erenfeucht-Mostowski type of the sequence, and hence is still true for the indiscernible sequence. Note that $f(a_i) = e$.
We have that $(a_i)_{i<\omega}$ is totally indiscernible over $C$ in the sense of $\ACF$ hence $a_1a_2a_3 \equiv^{\ACF}_C a_1a_3a_2$. Furthermore we have
$a_1\indi{\ACF}_{Ca_2} \quad a_3$, hence by \ref{INV} $a_1 \indi{\ACF}_{Ca_3} \quad a_2$.
By elimination of imaginaries in $\ACF$ it follows that $a_1\indi{\ACF}_{C}\quad  a_2 $, since $\ol{Ca}\cap \ol{Cb} = \ol{C}$. As $\pi(\overline{Ca_1})\cap \pi(\overline{Ca_2}) = \vect{\pi(C)\gamma}$, we have that $$a_1\indi w _{C\gamma} a_2.$$
As $f(a_1) = f(a_2) = e$, we deduce from Lemma~\ref{lm_critim} that $e\in \dcl^{eq}(C\gamma)$.
\end{proof}
 
\begin{ex}[$(K,K/G)$ does not eliminate finite imaginaries]\label{ex_notelim}
  The structure on $K/G$ is the one of an $\F_p$-vector space (with twisted algebraic and definable closures, $\acl(\alpha) = \vect{\pi(\overline{\F_p})\alpha}$ and $\dcl(\alpha) = \vect{\pi(\dcl(\F_p))\alpha}$). This follows from Corollaries~\ref{cor_typesinquo} and~\ref{cor_acleq}. Consider the unordered pair $\set{\alpha,\beta}$ for two singletons $\alpha, \beta \in K/G$, linearly independent over $\pi(\overline{\F_p})$. Assume that there exists a tuple $d\gamma$ such that for all automorphism $\sigma$ of $(K,K/G)$ $$\sigma(d\gamma) = d\gamma \iff \sigma(\set{\alpha,\beta}) = \set{\alpha,\beta}.$$
 As $d\gamma$ and $\alpha\beta$ are interalgebraic, we have first that $d\subset \overline{\F_p}$ and hence $\alpha,\beta\in \acl(\gamma)\cap K/G = \dcl(\gamma)\cap K/G = \vect{\gamma}$. As $\alpha,\beta$ are linearly independent over $\acl(\emptyset)$, we have $\alpha\beta\equiv_\emptyset \beta\alpha $ so let $\sigma$ be an automorphism of $(K,K/G)$ sending $\alpha\beta$ on $\beta\alpha$. As $\sigma$ fixes $\gamma$, it fixes $\vect{\gamma}$ hence $\alpha = \beta$, a contradiction.
 \end{ex}

 \begin{ex}[$K\times (K/G)^{eq}$ does not eliminate finite imaginaries]\label{ex_notelimst}
   Let $t$ be a transcendental element over $\F_p$. For convenience, we assume for that $G(\overline{\F_p(t)}) = \overline{\F_p(t)}$ (in a model $(K,G)$ of $\ACFG$ such that $G(\overline{\F_p}) = \overline{\F_p}$). Let $\alpha,\beta\in K/G$ be $\F_p$-independent, and let $e$ be the unordered pair $\set{\sqrt{t}\alpha, -\sqrt{t}\beta}$. We have the following:
 \begin{enumerate}
 \item $\dcl^{eq}(e)\cap K = \dcl(t)$
 \item $\dcl^{eq}(e)\cap (K/G)^{eq} = \dcl^{eq}(\set{\alpha,\beta})\cap (K/G)^{eq}$
 \end{enumerate}
 (1) The right to left inclusion is clear. Let $u\in \dcl^{eq}(e)\cap K$, in particular $u\in \dcl^{eq}(t,\alpha\beta)\cap K\subseteq \acl^{eq}(t,\alpha\beta)\cap K = \overline{\F(t)}$. Assume that $u\notin \dcl(t)$. There exists $u'\neq u $ with $u'\equiv_t u$. Let $\alpha',\beta'$ such that $u'\alpha'\beta'\equiv_t u\alpha\beta$. As $\alpha,\beta$ and $\alpha',\beta'$ are $\F_p$-lineary independent over $\pi(\overline{\F(t,u)})=\pi(\overline{\F(t)})=\set{0}$, we have that $\alpha\beta\equiv_{\overline{\F_p(t)}}\alpha'\beta'$ (Corollary~\ref{cor_typesinquo}). It follows that $u'\equiv_{t,\alpha,\beta} u$ hence $u' \equiv_e u$ so $u\notin \dcl^{eq}(e)$, a contradiction.\\
 (2) The right to left inclusion is clear. Let $\set{\gamma_1,\dots,\gamma_n}$ be an element of $\dcl^{eq}(e)\cap (K/G)^{eq}$. For all $i$, $\gamma_i$ is algebraic over $t\alpha\beta$, by Corollary~\ref{cor_acleq} $\gamma_i\in \vect{\pi(\overline{\F_p(t)}),\alpha,\beta} = \vect{\alpha,\beta}$. It follows that permutations of the set $\set{\sqrt{t}\alpha, -\sqrt{t}\beta}$ that permutes $\set{\gamma_1,\dots,\gamma_n}$ are exactly permutations of the set $\set{\alpha,\beta}$ that permutes $\set{\gamma_1,\dots,\gamma_n}$ hence $\set{\gamma_1,\dots,\gamma_n}\in \dcl^{eq}(\set{\alpha,\beta})$. In fact, such a set $\set{\gamma_1,\dots,\gamma_n}$ is the union of two sets of the same cardinal (possibly intersecting), every element in one set is of the form $\lambda\alpha+\mu\beta$ and has a ``dual'' element $\mu\alpha+\lambda \beta$ in the other set.
 
 If $e$ is interdefinable with an element from $K\times (K/G)^{eq}$, by \textit{(1)} and \textit{(2)}, we may assume that $e\in \dcl^{eq}(t\set{\alpha,\beta})$. By hypothesis $\alpha \beta \equiv_{\overline{\F_p(t)}} \beta\alpha$, hence an automorphism sending $\sqrt{t},-\sqrt{t}\alpha\beta$ to $\sqrt{t},-\sqrt{t}\beta\alpha$ fixes $t\set{\alpha,\beta}$ and moves $e$ to $\set{\sqrt{t}\beta, -\sqrt{t}\alpha}$, hence $e\notin \dcl^{eq}(t\set{\alpha,\beta})$, a contradiction.
\end{ex}

\section{Forking and thorn-forking in $\ACFG$}\label{sec_fork}

In this section, we give a description of forking and thorn-forking in the theory $\ACFG$. We also link these notions with other classical relations or other independence relations encountered in the previous chapters.

\subsection{Forcing base monotonicity and extension}\label{sub_forcing}

 In this subsection, given a ternary relation $\ind$ in an arbitrary theory, we introduce the relations $\indi{m} $ and $\indi *$, following the work of Adler in \cite{A09}. 

\begin{df}[Monotonised]
Let $\ind$ be any ternary relation, we define $\indi{m} $ to be the relation defined by
$$A\indi{m}_C  B \iff \forall D\subseteq CB \ \   A\ind_{CD} BC.$$
We call $\indi{m} $ the \emph{monotonised} of $\ind$.
\end{df}

Note that the relation $\indi M\ $ in~\cite[Section 4]{A09} is the relation ${\indi a}^{m}$ in our context.

\begin{lm}\label{lm_mon}
  The relation $\indi{m} $ satisfies \ref{BMON}. 
  Furthermore, for each of the following point
  \begin{itemize}
  \item \ref{INV}
  \item \ref{MON}
  \item \ref{TRA} 
\end{itemize}
if $\ind$ satisfies it then so does $\indi{m}$.
\end{lm}

\begin{proof}
Let $A,B,C,D$ such that $A\ind^{m}_C BD$. Then for all $D'\subseteq \acl(BCD)$ we have that $A\ind_{CD'} B$ so in particular for all $D'\subseteq \acl(BCD)$ containing $D$ we have $A\ind_{CD'} B$ hence for all $D''\subseteq \acl(BCD)$ we have $A\ind_{CDD''} B$ hence $A\ind^{m}_{CD} B$. To prove that \ref{INV} is preserved, note that there exists an isomorphism $\sigma : ABC \rightarrow A'B'C'$ which extends to $\acl(ABC) \rightarrow \acl(A'B'C')$ and so induces an isomorphism $ABCD\rightarrow A'B'C'\sigma(D)$ for all $D\subseteq \acl(BC)$. For \ref{MON}, it is an easy checking. For \ref{TRA} Assume that $B\indi{m}_C  A$ and $A'\indi{m}_{CB}  A$, and take $D\subseteq \acl(AC)$. We have in particular that $B\ind_{CD} A$ and $A'\ind_{CBD}A$ hence using \ref{TRA} of $\ind$ we have $A'B\ind_{CD} A$. This holds for any $D\subseteq \acl(AC)$ hence $A'B\indi{m}_C  A$.
\end{proof}

  Let $\ind$, $\ind'$ be two ternary relations, such that $\ind'$ is stronger than $\ind$. If $\ind'$ satisfies \ref{BMON} then $\ind'$ is stronger than $\indi{m}$. Note that $\ind$ may be symmetric and $\indi{m} $ not (see Corollary~\ref{cor_fork}). However in some cases, the monotonised is symmetric, as shows the following example.

\begin{ex}\label{ex_ACFmon}
We work here in $\ACF$. We have $$A{\indi a}^{m}_C\quad B \iff A\indi{\ACF}_C\quad B.$$
Indeed the right to left implication follows from $\indi{\ACF}\quad\rightarrow \indi a$ and the fact that $\indi{\ACF}\quad$ satisfies \ref{BMON}. 
From left to right, assume that $A\nindi{\ACF}_C\quad  B$, we may assume that $A,B,C$ are algebraically closed, and $C=A\cap B$. 
There exists $b_1,\dots,b_s\in B$ algebraically independent over $C$ such that for $D=\set{b_2,\dots,b_s}$, then we have $b_1\in (\overline{AD}\cap B) \setminus \overline{CD}$ so $A{\nindi a}^{m}_C B$. \\

This result translates as follows: in $\ACF$, $\indi f = {\indi a }^{m}$. It raises the following question: when do we recover forking independence from the monotonised of the relation $\indi a$? Does the \ref{SYM} of the monotonised of a symmetric relation imply nice features on the theory? Observe that the proof above shows that in any pregeometry $(S,\cl)$, the independence relation associated with the pregeometry is obtained by forcing \ref{BMON} on the relation $A \ind _C B\iff \cl(AC)\cap \cl(BC) = \cl(C)$.
\end{ex}

The following example shows that the monotonised does not preserve \ref{LOC}. Also it implies that $\indi{st}$ doesn't satisfy \ref{LOC} since $\indi{st}\rightarrow {\indi w}^{m}$.

\begin{ex}\label{ex_monloc}
  In $\ACFG$, the relation $\indi{w}^{m}$ does not satisfy \ref{LOC}.\\
  Let $\kappa$ be any uncountable cardinal and consider the set $A = \set{t_i,t_i'\mid i<\kappa}$ and an element $t$ such that $t(t_i,t_i')_{i<\kappa}$ are algebraically independent over $K$. Let $F = \overline{\F_p(t,A)}$ and define $H$ over $F$ as $G(\overline{\F_p})+\vect{t\cdot t_i+t_i'\mid i<\kappa}$. The pair $(F,H)$ defines a consistent type over $\emptyset$, as $\overline{\F_p}\cap H = G(\overline{\F_p})$ and $F\cap K = \overline{\F_p}$, so we assume that $t,A$ are realisation of the type in $K$.
By contradiction suppose that there exists $A_0\subset A$ with $\abs{A_0}\leq \aleph_0$ such that $t\indi{w}^{m}_{A_0} A$. By definition, for all $D\subseteq A$ we have $t\indi{w}_{A_0D} A$. Let $D = \set{t_i\mid i<\kappa}\setminus A_0$. We have that 
$$G(\overline{tDA_0} + \overline{A}) = G(\overline{tDA_0}) + G(\overline{A}).$$
We compute the $\F_p$-dimension over $G(\overline{\F_p})$ on each side of the previous equation. On one hand, we have $t\cdot t_i + t_i'\in G(\overline{tDA_0} + \overline{A}) $ for all $i<\kappa$, as they are $\F_p$-linearly independent over $\overline{\F_p}$ we have $\F_p$-$\dim(G(\overline{tDA_0} + \overline{A})/G(\overline{\F_p})) \geq \kappa$. 
For all $i<\kappa$, $t\cdot t_i +t_i'\in G(\overline{tDA_0})$ if and only if $t_i'\in \overline{tDA_0}$ if and only if $t_i'\in A_0$, because if $t_i'$ is algebraic over $t, A_0,t_1,\dots,t_k$ then $t$ is in $A_0$ otherwise this contradicts that $t ,A$ are algebraically independent.
We conclude that $\F_p$-dim$(G(\overline{tDA_0})/G(\overline{\F_p})) \leq \abs{A_0}\leq \aleph_0$. As $G(\overline{A}) = G(\overline{\F_p})$ we have that $\F_p$-dim$([G(\overline{tDA_0}) +G(\overline{A})]/G(\overline{\F_p})) \leq \aleph_0$ so the equality cannot hold.
\end{ex}

\begin{df}[Adler,~\cite{A09} Section 3] For $\ind$ any ternary relation, $\indi{*}$ is defined as follows:
$$A\indi{*}_C B \iff \forall \hat{B}\supseteq B\  \exists A'\equiv_{BC}A \ A'\ind_C \hat B.$$
 \end{df}
 
\begin{fact}[~\cite{A09} Lemma 3.1]\label{fact_ext}
 If $\ind$ satisfies \ref{INV} and \ref{MON} then $\indi *$ satisfies \ref{INV}, \ref{MON} and \ref{EXT2}. Furthermore, for each of the following point
  \begin{itemize}
  \item \ref{BMON}
    \item \ref{TRA} 
 \item \ref{EXT} 
 \end{itemize}
if $\ind$ satisfies it then so does $\indi{*}$. 
\end{fact}
 
Recall that $a\indi u _C b$ if and only if $tp(a/Cb)$ is finitely satisfiable in $C$.

\begin{rk}\label{lm_uf}
Let $(b_i)_{i<\kappa}$ be a $C$-indiscernible infinite sequence with $\kappa>\omega$. Then for all $\ \geq \alpha\geq \omega$ $$b_{<\beta} \indi u _{Cb_{<\alpha}}b_\beta .$$
Furthermore, for $\kappa$ big enough, the sequence $(b_i)_{i<\kappa}$ is indiscernible over $\acl(C)$ (see ~\cite[Corollary 1.7, 2.]{C11}).
\end{rk}

\begin{lm}\label{lm_clo}
  If $\ind $ satisfies \ref{INV} and \ref{EXT2}, then $A\ind_C B$ implies $A\ind _C \acl(CB)$. If $\ind$ further satisfies \ref{BMON}, then $\ind$ satisfies \ref{CLO}.
\end{lm}
\begin{proof}
  Assume that $A\ind _C B$. By \ref{EXT2}, let $A'$ be such that $A'\equiv_{BC} A$ and $A'\ind _C \acl(BC)$. There is an automorphism $\sigma$ over $BC$ sending $A'$ to $A$ hence by \ref{INV}, $A\ind _C \sigma(\acl(BC))$. Now, as sets, $\sigma(\acl(BC)) = \acl(BC)$ so $A\ind_C \acl(BC)$. The last assertion is trivial, as $\acl(C)\subseteq \acl(BC)$.
\end{proof}
\begin{rk}\label{rk_clomon}
  By Lemma~\ref{lm_clo} and Fact~\ref{fact_ext}, if $\ind$ satisfies \ref{INV}, \ref{MON}, then $\indi *$ satisfies \ref{INV}, \ref{MON}, \ref{EXT2} and \ref{CLO} over algebraically closed sets. If $\ind$ satisfies also \ref{BMON}, then so does $\indi *$ hence $\indi *$ satisfies \ref{CLO} over any sets. In particular, by Lemma~\ref{lm_mon} if $\ind$ satisfies \ref{INV} and \ref{MON}, then ${\indi m\ }^*$ satisfies \ref{INV}, \ref{MON}, \ref{CLO}, \ref{BMON}, \ref{EXT2}. 
\end{rk}

\begin{fact}\label{fact_ind}
  The following are standard facts more or less obvious from the definition.
  \begin{enumerate}
   \item $\indi a$ satisfies \ref{INV}, \ref{MON}, \ref{TRA}, \ref{EX}, \ref{EXT2} and \ref{EXT};
   \item $\indi d$ satisfies \ref{INV}, \ref{MON}, \ref{BMON}, \ref{TRA};
   \item $\indi f$ satisfies \ref{INV}, \ref{MON}, \ref{BMON}, \ref{TRA} and \ref{EXT2};
   \item $\indi u$ satisfies \ref{INV}, \ref{MON}, \ref{BMON}, \ref{TRA}, \ref{EXT2}, \ref{EX} over models, \ref{EXT} over models;
 \item $\indi d\rightarrow \indi{aeq}\ \upharpoonright \M$;
 \item $\indi u \rightarrow \indi f \rightarrow \indi d \rightarrow \indi{aeq}\ \upharpoonright \M \rightarrow \indi a$;
 \item $\indi f\rightarrow \indi K$ and $\indi d \rightarrow \indi{Kd}$.
 \end{enumerate}
\end{fact}
\begin{proof}
  \textit{(1)} is~\cite[Proposition 1.5]{A09}. \textit{(2)} and \textit{(3)} are~\cite[Proposition 1.3]{Ad09}. \textit{(4)} is~\cite[Remark 2.16]{CK12}, \ref{BMON} is trivial. For \textit{(5)}, it is clear that if $a\indi d _C b$ in $\M$, then $a\indi d _C b$ in $\M^{eq}$, and by~\cite[Remark 5.4]{A09} it follows that $\acl^{eq}(Ca)\cap \acl^{eq}(Cb) = \acl^{eq}(C)$ hence $a\indi{aeq}_C\ b$. \textit{(6)} follows from~\cite[Example 2.22]{CK12}, and the previous results. \textit{(7)} is by definition.
\end{proof}

\begin{lm}\label{lm_fork}
  Let $\ind$ be a ternary relation, which satisfies 
  \begin{itemize}
    \item \ref{INV}, \ref{MON};
  \item $\indi u$-\ref{AM} over algebraically closed sets.
  \end{itemize}Then ${\indi{m}\ } ^{*} \rightarrow \indi f$. 
\end{lm}
  \begin{proof}
    We show that ${\indi m \ }^* \rightarrow \indi d$, the result follows from $\indi f = {\indi d}^*$, in Fact~\ref{fact_ind}. By Remark~\ref{rk_clomon}, ${\indi{m}\ } ^{*}$ satisfies \ref{INV}, \ref{MON}, \ref{BMON}, \ref{EXT2} and \ref{CLO}. Assume $a{\indi{m}\ } ^{*} _C b$, for any $a,b,C$. Let $(b_i)_{i<\kappa}$ be a $C$-indiscernible sequence with $b = b_0$, for a big enough $\kappa$. By Remark~\ref{lm_uf}, $b_{<i}\indi u _{Cb<\omega} b_i$ for all $i\geq \omega$. By Fact~\ref{fact_ind}, and Lemma~\ref{lm_clo}, $\indi u $ satisfies \ref{CLO} and \ref{MON}, hence $b_{<i}\indi u _{\acl(Cb<\omega)} b_i$. Also $(b_i)_{i\geq\omega}$ is $Cb_{<\omega}$-indiscernible, so if $\kappa$ is big enough, by Remark~\ref{lm_uf} we have that $b_i\equiv_{\acl(Cb_{<\omega})} b_\omega$. 
    There exists a $C$-automorphisme sending $b$ to $b_{\omega}$ hence there exists some $a_\omega$ such that $a_\omega b_\omega\equiv_{C} ab$. By \ref{INV}, we have $a_\omega {\indi{m}\ }^{*}_C b_\omega $, so by \ref{CLO} we have $a_\omega {\indi m\ }^{*}_{\acl(C)} \acl(Cb_\omega)$, hence by \ref{EXT2} there exists $a_\omega'$ such that $a_\omega'\equiv_{\acl(Cb_\omega)} a_\omega$ and $a_\omega' {\indi{m}\ }^*_{\acl(C)}   b_\omega b_{<\omega}$. It follows from \ref{CLO} and \ref{BMON} that
    $$a_\omega' \ind_{\acl(Cb_{<\omega})} b_\omega.$$
We also have $$a_\omega'b_\omega \equiv_C a_\omega b_\omega \equiv_C ab.$$
For each $i\geq \omega$ there exists an $\acl(Cb_{<\omega})$-automorphism $\sigma_i$ sending $b_\omega$ to $b_i$, so setting $a_i' = \sigma_i(a_\omega')$ we have:
$$\forall i\geq \omega\ \ a_i'b_i \equiv_{\acl(Cb_{<\omega})} a_\omega'b_\omega \text{ and } a_i'\ind_{\acl(Cb_{<\omega})} b_i.$$

We show that there exists $a''$ such that $a''b_i \equiv_{\acl(Cb_{<\omega})} a_\omega b_\omega$ for all $\omega \leq i<\omega+\omega$. By induction and compactness, it is sufficient to show that for all $\omega\leq i<\omega+\omega$, there exists $a_i''$ such that for all $\omega \leq k\leq i$ we have $a''_i b_k \equiv_{\acl(Cb_{<\omega})} a_\omega b_\omega$ and $a_i''\ind_{\acl(Cb_{<\omega})} b_{\leq i}$. For the case $i = \omega$ take $a_\omega'' = a_\omega'$. Assume that $a_i''$ has been constructed, we have 
$$ a_{i+1}' \ind_{\acl(Cb_{<\omega})}b_{i+1}\text{ and } b_{\leq i} \indi{u}_{\acl(Cb_{<\omega})} b_{ i+1} \text{ and }  a_i''\ind_{\acl(Cb_{<\omega})}b_{\leq i}.$$

As $a_{i+1}' \equiv_{\acl(Cb_{<\omega})} a_i''$, by $\indi{u}$-\ref{AM} over algebraically closed sets, there exists $a_{i+1}''$ such that 
\begin{enumerate}
    \item $a_{i+1}''b_{i+1} \equiv_{\acl(Cb_{<\omega})} a_{i+1}'b_{i+1}$
    \item $a_{i+1}'' b_{\leq i}\equiv_{\acl(Cb_{<\omega})} a_i''b_{\leq i}$
    \item $a_{i+1}'' \ind_{\acl(Cb_{<\omega})} b_{\leq i+1}$.
\end{enumerate}
By induction and compactness there exists $a''$ be such that $a''b_i \equiv_{\acl(Cb_{<\omega})} a_\omega b_\omega$ for all $\omega \leq i<\omega+\omega$. By indiscernibility of $(b_i)_{i<\kappa}$ there exists $a''' $ such that for all $i<\kappa$ $a'''b_i \equiv_C ab$, hence $a\indi{d}_C b$.
\end{proof}

\begin{rk}
  It is important to observe that since $\indi u$ is not in general a symmetric relation, the parameters $a$ and $b$ in the statement of $\indi u$-\ref{AM} do not play a symmetrical role. If a relation satisfies $\indi u$-amalgamation, we mean that $tp(c_1/Ca)$ and $tp(c_2/Cb)$ can be amalgamated whenever $a\indi u _C b$ or $b\indi u _C a$.
\end{rk}

\begin{prop}\label{prop_fork}
Let $\ind$ be a relation such that
\begin{enumerate}
\item $\ind$ is weaker than $\indi d$;
\item $\ind$ satisfies \ref{INV}, \ref{MON}, $\indi u$-\ref{AM} over algebraically closed sets;
\item ${\indi m}$ satisfies \ref{EXT2} over algebraically closed sets;
\end{enumerate}
Then ${\indi m} = \indi f = \indi d$. 
\end{prop}

\begin{proof}
The relation $\indi d$ satisfies \ref{BMON} by Fact~\ref{fact_ind} hence from \textit{(1)} we have $\indi d \rightarrow {\indi m}$. By hypothesis \textit{(3)}, ${\indi m}={\indi m\ }^*$, hence by \textit{(2)} and Lemma~\ref{lm_fork} we have $\indi d = {\indi m}=\indi f$.
\end{proof}

%
%

\subsection{Forking in $\ACFG$} 

We show that forking in $\ACFG$ is obtained by forcing the property \ref{BMON} on Kim-independence. 

We work in a big model $(K,G)$ of $\ACFG$.

\begin{lm}\label{lm_modgp}
  Let $A,B,C$ be three additive subgroups of $K$, then $A\cap (B+C) = A\cap \left[B+ C\cap(A+B)\right]$. 
\end{lm}
\begin{proof}
  Let $a \in A\cap (B+C)$. There exist $b\in B$ and $c\in C$, such that $a = b+c$. Then $c = a-b\in C\cap (A+B)$ hence $a\in A\cap \left[B+ C\cap(A+B)\right]$. The other inclusion is trivial. 
\end{proof}

\begin{lm}[Mixed Transitivity on the left]\label{mixed_tran}
Let $A,B,C,D$ be algebraically closed sets, with $A,B,D$ containing $C$ and $B\subseteq D$. If $A\indi{w}^{m}_C\ \ \  B$ and  $A\indi{st}_B D $ then $A\indi{w}^{m}_C\ \ \  D$.
\end{lm}

\begin{proof}
Let $A,B,C,D$ be as in the hypothesis. Let $E\subseteq D$ containing $C$, we want to show that $A\indi{w}_E D$. We may assume that $E$ is algebraically closed. We clearly have $A\indi{\ACF}_E\quad  D$, so we have to show that $$G(\overline{AE} + D) = G(\overline{AE})+G(D).$$
From $A\indi{\ACF}_C \quad  E,B$ we have $\overline{AE}\cap \overline{AB}\indi{\ACF}_E\  E,B$ and $\overline{AE}\cap \overline{AB}\indi{\ACF}_B\ E,B$. By elimination of imaginaries in $\ACF$, $\overline{AE}\cap \overline{AB}\indi{\ACF}_{E\cap B} E,B$. By Lemma~\ref{lm_CH}, it follows that $\overline{AE}\cap \overline{AB} = \overline{A(E\cap B)}$. \\
\begin{claim} 
  $(\overline{AE}+D)\cap (\overline{AB}+D) = \overline{A(E\cap B)}+D$.
\end{claim}
\begin{proof}[Proof of the claim]
  By modularity, we have that $(\overline{AE}+D)\cap (\overline{AB}+D) = D + \ol{AE}\cap (\ol{AB}+D)$. By Lemma~\ref{lm_modgp} we have that 
  $$\ol{AE}\cap (\ol{AB}+D) = \ol{AE}\cap \left(\ol{AB}+(\ol{AE}+\ol{AB})\cap D\right).$$
  Applying~\cite[Lemma 5.16.]{dE18A} with $T=\ACF$, we have $(\overline{AE}+\overline{AB})\cap D = E+B$, hence 
\begin{align*}
  \ol{AE}\cap (\ol{AB}+D) &= \ol{AE}\cap(\ol{AB}+E+B) \\
  &=\ol{AE}\cap (\ol{AB} + E)\\
  &= \ol{AE}\cap \ol{AB} + E\text{ by modularity }\\
  &= \ol{A(E\cap B)}+E.
\end{align*}
It follows that $(\overline{AE}+D)\cap (\overline{AB}+D) = \ol{A(E\cap B)} + D +E = \ol{A(E\cap B)}+D$.
\end{proof}

By hypothesis, $G(\overline{AD}) = G(\overline{AB})+G(D)$, so, by the claim $$G(\overline{AE} + D) = G(\overline{AE} + D)\cap (G(\overline{AB})+G(D)) = G\left(\overline{A(E\cap B)}+D\right)\cap G(\overline{AB})+G(D).$$
Furthermore $G\left(\overline{A(E\cap B)}+D\right)\cap G(\overline{AB}) =G\left(\overline{A(E\cap B)} + D\cap\overline{AB}\right) = G\left(\overline{A(E\cap B)} +B\right)$. As $A\indi{w}^{m}_C   B$ we have $G(\overline{A(E\cap B)} +B) = G(\overline{A(E\cap B)}) +G(B)$. We conclude that
$$G(\overline{AE} + D) = G(\overline{A(E\cap B)}) +G(B) + G(D) = G(\overline{A(E\cap B)})+ G(D).$$
\end{proof}

\begin{cor}\label{cor_fork}
In $\ACFG$, ${\indi{w}}^{m}$ satisfies \ref{EXT2}. In particular, in ${\indi{w}}^{m}=\indi f = \indi d$.
\end{cor}
\begin{proof}
  Assume that $a\indi{w}^{m}_C b$ and $d$ is given. By \ref{EXT} of $\indi{st}$ there exists $a'\equiv_{Cb} a$ such that $a'\indi{st}_{Cb} d$. Also $a'\indi{w}^{m}_C b$ hence by Lemma~\ref{mixed_tran} $a'\indi{w}^{m}_C b,d$, which shows \ref{EXT2} for ${\indi{w}}^{m}$. In particular $\indi w$ satisfies hypothesis \textit{(3)} of Proposition~\ref{prop_fork}. We check that it satisfies the rest of the hypotheses of Proposition~\ref{prop_fork}. \textit{(1)} follows from Corollary~\ref{cor_kim}. From Theorem~\ref{thm_indw}, $\indi w$ satisfies the properties \ref{INV}, \ref{MON} and $\indi u$-\ref{AM} over algebraically closed sets (since $\indi u\rightarrow \indi a$, by Fact~\ref{fact_ind}), so $\indi w$ satisfies \textit{(2)}.
\end{proof}

\subsection{Thorn-Forking in $\ACFG$}\label{sec_thornfork}
Let $(K,G)$ be a monster model of $\ACFG$.
Recall that $\indi{aeq}\quad$ is the relation $\indi a$ in the sense of $(K,G)^{eq}$. The thorn-forking independence relation $\indi{\text{\thorn}}$ is the relation defined over subsets of $(K,G)^{eq}$ by $\indi{\text{\thorn}} = {{(\indi{aeq}\quad  )}^{m}}^*$. We will only consider the restrictions of $\indi{aeq}\ $ and $\indi{\text{\thorn}}$ to the home sort, which we denote respectively by $\indi{aeq}\ \upharpoonright K$ and $\indi{\text{\thorn}}\upharpoonright K$. By Corollary~\ref{cor_acleq} and Theorem~\ref{thm_wei}, for $a,b,C\subset K$
$$a\indi{aeq} _C\ \ b \iff \overline{Ca}\cap \overline{Cb} = \overline{C} \text{ and } \pi(\overline{Ca})\cap \pi(\overline{Cb}) = \pi(\overline{C}).$$

\begin{fact}[~\cite{A09} Theorem 4.3]\label{fact_rosy}
The following are equivalent.
\begin{itemize}
    \item $T$ is rosy
    \item $\indi{\text{\thorn}}$ in $T^{eq}$ satisfies \ref{LOC}.
\end{itemize}
\end{fact}

\begin{prop}
Let $(K,G)$ be a model of $\ACFG$. Then $\indi{\text{\thorn}}\upharpoonright K = {\indi w}^{m} = \indi f = \indi d$. In particular $\ACFG$ is not rosy.
\end{prop}

\begin{proof}
Assume that $a\indi{\text{\thorn}}_C b$. In particular $a{\indi{aeq}\quad }^{m}_C b$ so for all $C\subseteq D \subseteq \overline{Cb}$ we have $\overline{Da}\cap \overline{Cb} = \overline{D}$ hence by Example~\ref{ex_ACFmon} we have 
$$a\indi{\ACF}_C\quad b.$$
On the other hand, we have $\pi(\overline{Ca})\cap \pi(\overline{Cb}) = \pi(\overline{C})$, hence by Section~\ref{subsec_im}  
$$a\indi{w}_C\quad b.$$
It follows that $\indi{\text{\thorn}}\upharpoonright K \rightarrow {\indi w}^{m}$. By Fact~\ref{fact_ind}, $\indi d\rightarrow \indi{aeq}\ \upharpoonright K$, hence as $\indi f$ satisfies \ref{BMON} and \ref{EXT2} it follows that $\indi f \rightarrow \indi{\text{\thorn}}\upharpoonright K$. Hence by Corollary~\ref{cor_fork} we conclude that $\indi{\text{\thorn}}\upharpoonright K = {\indi w}^{m} = \indi f = \indi d$. As $\ACFG$ is not simple, $\indi f$ does not satisfy \ref{LOC}, so $\indi{\text{\thorn}}\upharpoonright K$ does not satisfy \ref{LOC} hence neither does $\indi{\text{\thorn}}$. By Fact~\ref{fact_rosy}, $\ACFG$ is not rosy.
\end{proof}

\begin{rk}
There is another way of proving that $\ACFG$ is not rosy which does not use the description of forking in $\ACFG$ but only the fact that $\indi{\text{\thorn}}\upharpoonright K \rightarrow {\indi w}^{m}$. Indeed ${\indi w}^{m}$ does not satisfy \ref{LOC} from Example~\ref{ex_monloc} hence neither does $\indi{\text{\thorn}}\upharpoonright K$ and hence neither does $\indi{\text{\thorn}}$.
\end{rk}

\begin{rk}
It is worth mentioning that in the definition of $\indi{\text{\thorn}}$, the relation $\indi{aeq}\quad$ cannot be replaced by $\indi{a}$. Indeed, in the structure $(K,G)$, by Example~\ref{ex_ACFmon} ${\indi a }^{m} = \indi{\ACF}\quad$ and then as \ref{EXT2} clearly holds for $\indi{\ACF}\quad$, we have ${\indi{a} }^{m*}=\indi{\ACF}\quad$. This relation satisfies \ref{LOC}. This means that ${\indi{a} }^{m*}$ is not the restriction of ${{\indi{aeq}\quad }^{m}}^*$ to the home sort. This is what Adler mention in \cite[Example 4.5]{A09}.
\end{rk}

\subsection{Forking and thorn-forking in other generic constructions}\label{sec_forkdiscuss}

~

\noindent \textbf{Forking and dividing}. In the three following examples:\textit{
\begin{enumerate}
\item Generic $\LL$-structure $T^{\emptyset}_\LL$~\cite[Proposition 3.18]{KR17};
\item Generic $K_{n,m}$-free bipartite graph~\cite[Corollary 4.12]{CKr17};
\item omega-free PAC fields~\cite[Theorem 3.3]{C02};
\end{enumerate}}
\noindent we also have that forking and dividing coincides for types, and coincides with the monotonised of Kim-independence.
In \textit{(1)} and \textit{(2)} the strategy is the following: first prove that $\indi d = {\indi K}^{m}$ and then show that $\indi d$ satisfies \ref{EXT2}. The latter is obtained using \ref{EXT} of the \emph{strong} independence relation and a similar mixed transitivity result. This is discussed in~\cite[Subsection 3.3]{KrR18}. We followed a close strategy: using Lemma~\ref{lm_fork} (based on the approach of \textit{(3)}), have that ${\indi w }^{m*}$ strengthens $\indi d$. Then we use a mixed transitivity result and \ref{EXT} of the strong independence to show that ${\indi K}^{m}$ satisfies \ref{EXT2}. These results suggest that Proposition~\ref{prop_fork} can be used to show that in other examples of $\NSOP 1$ theories, forking and dividing agrees on types, for instance in Steiner triple system~\cite{BC18}, or bilinear form over an infinite dimensional vector space over an algebraically
closed field~\cite{G99}~\cite{CR16}.\\

\noindent\textbf{Strong independence and Mixed Transitivity}. There is also a notion of \emph{strong independence} in the three previous examples 
which is symmetric and stationary over algebraically closed sets.  Concerning \textit{(3)} the strong independence satisfies also the other axioms for mock stability~\cite[Example 0.1 (3)]{KK11}. In \textit{(2)}, it also satisfies \ref{EXT}, \ref{MON} and \ref{TRA}~\cite[Proposition 4.20]{CKr17}. In \textit{(1)}, it is defined in~\cite[Remark 3.19]{KrR18}, as a remark, to state a mixed transitivity result, but nothing about it is proven. It is likely that \textit{(1)} and \textit{(2)},  are also mock stable, witnessed by the strong independence. Informally, the strong independence is in general defined to hold between two sets when they are the most unrelated to each other with respect to the ambient theory. Another way of seeing this relation is by saying that the two sets can be somehow ``\emph{freely amalgamated}''. The definition given in~\cite[Remark 3.19]{KR17} make this precise, for $C\subseteq A\cap B$, we have $A\indi{\otimes}_C B$ if and only if the substructure spanned by $ABC$ is isomorphic to the fibered coproduct of the structures spanned by $A$ and $B$ over the substructure spanned by $C$. This definition coincides with our definition of strong independence in $\ACFG$.
\begin{question}\label{Qst}
  Is there a model-theoretic definition of the strong independence that encompasses the strong independence in the three examples above and in $\ACFG$? 
\end{question}

The mixed transitivity result (Lemma~\ref{mixed_tran}) is starting to be reccurent in $\NSOP 1$ examples. It holds in example \textit{(1)} (\cite[Remark 3.19]{KrR18}) and in \textit{(2)} (\cite[Lemma 4.23]{CKr17}). Note that a similar mixed transitivity appears in a $SOP_3$ (hence $SOP_1$) example: the generic $K_n$-free graph (\cite{C17}), this was observed in \cite[Remark 3.19]{KrR18}.

The mixed transitivity result holds as well in omega-free PAC fields. Let $\indi w$ be the \emph{weak} independence and $\indi{st}$ the \emph{strong} independence in the sens of~\cite[(1.2)]{C02}. Then for all $A,B,C,D$ $\acl$-closed in an omega free PAC field, with $C\subseteq A\cap B$ and $B\subseteq D$ we have:
\begin{center}If $A{\indi w}^{m} _C B $ and $A\indi{st}_B D$ then $A{\indi w}^{m} _C D$. \end{center}
This is contained in the proof\footnote{In the proof of~\cite[(3.1) Proposition]{C02}, $D$ contains $B$, $\psi$ is over $C$ and $F\cap (C\psi(D))^s = C\psi(D)$, hence $\psi(D)$ and $C$ satisfies condition \textit{(I3)} over $B$, so $A_1 = \psi(A_0)$ and $C$ satisfies condition \textit{(I3)} over $E$. As $A_1$ and $C$ satisfies condition \textit{(I1)} over $E$, $A_1$ and $C$ are strongly independent over $E$. Also $A_1$ and $B$ satisfy condition \textit{(I1)} and \textit{(I2)} over $E$. The rest of the proof consist in proving that $A_1$ and $C$ satisfy condition \textit{(I2)} over $E$.} of~\cite[(3.1) Proposition]{C02}. \\

\noindent \textbf{Thorn-forking}. The three other examples are also not rosy. For \textit{(1)}, it is~\cite[Subsection 3.3]{KR17}, for \textit{(2)}, it is \cite[Proposition 4.28]{CKr17} and for \textit{(3)}, it is~\cite[Subsection 3.5]{C08}. Also, for both \textit{(1)} and \textit{(2)} we have $\indi f = \indi d = \indi{\text{\thorn}}$, and they both weakly eliminate imaginaries. 

The following questions have been asked for the last two or three years by specialists in regards to the observations above.
\begin{question}
  \begin{enumerate}
    \item[$(Q_1)$] Does forking equals dividing for types in every $\NSOP 1$ theory?
    \item[$(Q_2)$] Does the mixed transitivity result holds in every $\NSOP 1$ theory? (Provided an answer to Question~\ref{Qst}.)
    \item[$(Q_3)$] Is there an $\NSOP 1$ not simple rosy theory?
  \end{enumerate}
\end{question}

\begin{rk}\label{rk_sM}
In omega-free PAC fields~\cite{C02}, the strong independence $\indi{st}$ and the weak independence $\indi{w}$ are linked by the following relation for $A,B,C$ $\acl$-closed, $A\cap B = C$:
$$A\indi{st}_C B \iff \text{for all $C\subseteq D\subseteq A$ and $C\subseteq D'\subseteq B$ } A\indi{w}_{DD'} B.$$
In $\ACFG$ this is not the case. Let $(K,G)$ be a model of $\ACFG$ and for conveniance assume that $G(\overline{\F_p}) = \set{0}$. Let $t$ and $t'$ be algebraically independent over $\F_p$, let $u = t\cdot t'$. Assume that $G(\overline{\F_p(t,t')}) = \vect{u}$. Then by~\cite[Lemma 5.19.]{dE18A}, $u\notin \overline{\F_p(t)}+\overline{\F_p(t')}$, so $G(\overline{\F_p(t)})+G(\overline{\F_p(t')})=\set{0}$ so $t\nindi{st} \ t'$. We show that for all $D\subseteq \overline{\F_p(t)}$ and $D'\subseteq \overline{\F_p(t')}$ we have $t\indi{w}_{DD'} t'$. Let $D$ and $D'$ be  as such. There are three cases to consider (the middle case is symmetric):
\begin{align*}
&t\cdot t' \in \overline{D't} \text{ and } t\cdot t' \in \overline{Dt'} \quad &G(\overline{D't}) = \vect{u}\quad &G(\overline{Dt'}) = \vect{u} \quad &G(\overline{D't}+\overline{Dt'}) = \vect{u}\\
&t\cdot t' \in \overline{D't} \text{ and } t\cdot t' \notin \overline{Dt'} \quad &G(\overline{D't}) = \vect{u}\quad &G(\overline{Dt'}) = \set{0} \quad &G(\overline{D't}+\overline{Dt'}) = \vect{u}\\
&t\cdot t' \notin \overline{D't} \text{ and } t\cdot t' \notin \overline{Dt'} \quad &G(\overline{D't}) = \set{0}\quad &G(\overline{Dt'}) = \set{0} \quad &G(\overline{D't}+\overline{Dt'}) = \set{0}
\end{align*}
In every cases we have $G(\overline{D't}+\overline{Dt'}) = G(\overline{D't}) + G(\overline{Dt'})$. As $t\indi{\ACF}_{DD'} t'$ is clear we have $t\indi{w}_{DD'} t'$.
\end{rk}

\noindent \textbf{Summary on independence relations in $\ACFG$}.
Every arrow in Figure~\ref{fig_diag2} is strict, from that point of view, $\ACFG$ is different from \textit{(1)}, \textit{(2)} and \textit{(3)}.
\begin{figure}
  \begin{center}
    \includegraphics{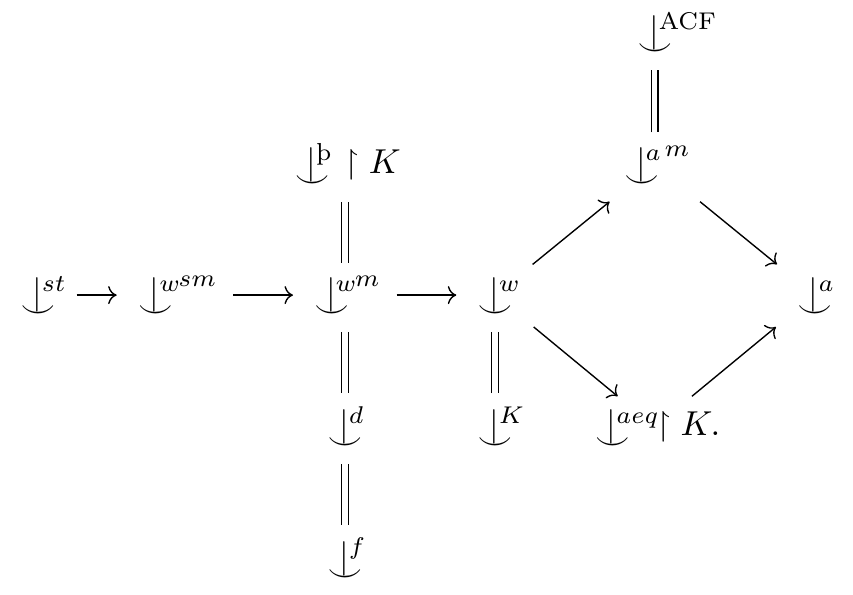}
  \end{center}
\caption{Interactions of independence relations in $\ACFG$.}\label{fig_diag2}
\end{figure}

Denote by $A\indi{w}^{sm}_C\quad B$ the relation for all $C\subseteq D\subseteq \overline{AC}$ and $C\subseteq D'\subseteq \overline{BC}$  $A\indi{w}_{DD'} B$. Remark~\ref{rk_sM} states that $\indi{st}$ is strictly stronger than $\indi{K}^{sm}$, in \textit{(3)}, this is not the case. In \textit{(1)}, we have that $\indi a =\indi{aeq} \ = \indi K$ is strictly weaker than ${\indi a}^{m} = \indi d = \indi f = \indi{\text{\thorn}}$. In \textit{(2)}, $\indi a =\indi{aeq}\ $ is strictly weaker than $\indi{K}$ and ${\indi K}^{m} = \indi d = \indi f = \indi{\text{\thorn}}$. \\

\textbf{Acknowledgements}. This work is part of the author's Ph.D. dissertation. The author would like to thank Thomas Blossier and Zoé Chatzidakis for their precious advising. The author is very grateful to Alex Kruckman and Nicholas Ramsey for numerous exchanges on  $\NSOP 1$ theory and examples. The author would like to give a special thanks to Gabriel Conant for many useful discussions. Part of the result in Subsection~\ref{sub_forcing} was also observed by Gabriel Conant in an unpublished note which inspired the form of this subsection.

\bibliographystyle{plain}
\bibliography{biblio}

\end{document}